\documentclass[12pt]{article}
\usepackage{amsfonts}
\usepackage{natbib}
\usepackage{graphicx}
\usepackage{latexsym,amsmath,color,esint}
\usepackage[english]{babel}

\def\1{\mathbb{I}}

\newcounter{thm}[section]
\newcounter{appen}[section]

\newtheorem{proposition}[thm]{Proposition}
\newtheorem{assumption}[appen]{Assumption}
\newtheorem{lem}[thm]{Lemma}

\newenvironment{proof}[1][Proof]{\noindent \textbf{#1.}
}{\rule{0.5em}{0.5em}}

\begin{document}

\title{Nonparametric model checks of single-index assumptions}
\author{Samuel Maistre\footnote{Corresponding author. CREST (Ensai) \& IRMAR (UEB), France; samuel.maistre@ensai.fr} \quad and \quad
Valentin Patilea\footnote{CREST (Ensai) \& IRMAR (UEB), France; patilea@ensai.fr}}
\date{\today}

\maketitle

\begin{abstract}
\smallskip

Semiparametric single-index assumptions are convenient and widely used dimen\-sion reduction approaches that represent a compromise between the parametric and fully nonparametric models for regressions or  conditional laws. In a mean regression setup, the SIM assumption means that the conditional expectation of the response given the vector of covariates is the same as the conditional expectation of the response given a scalar projection of the covariate vector. In a conditional distribution modeling, under the SIM assumption the conditional law of a response given the covariate vector coincides with the conditional law given a linear combination of the covariates.
Several estimation techniques for single-index models are available and commonly used in applications. However, the problem of testing the goodness-of-fit seems less explored and the existing proposals still have some major drawbacks.
In this paper, a novel kernel-based approach for testing SIM assumptions is introduced. The covariate vector needs not have a density and only the index estimated under the SIM assumption is used in kernel smoothing. Hence the effect of high-dimensional covariates is mitigated while asymptotic normality of the test statistic is obtained.
Irrespective of the fixed dimension of the covariate vector, the new test
detects local alternatives approaching the null hypothesis slower than
$n^{-1/2}h^{-1/4},$ where $h$ is the bandwidth used to build the test statistic and $n$ is the sample size.
A wild bootstrap procedure is proposed for finite sample corrections of the asymptotic critical values. The small sample performances of our test compared to existing procedures are illustrated through simulations.
\smallskip

\end{abstract}

\newpage


%
%
%
%
%
%
%
%
%


\section{Introduction}

Semiparametric single index models (SIM) are widely used tools for statistical modeling.
The paradigm of such models is based on the assumption that the information contained in a vector of conditioning random variables is equivalent, in some sense, to the information contained in some index, that is usually a linear combination of the vector components. This assumption underlies most of the statistical parametric models including covariates, but allows for more general semiparametric modeling.
The most common semiparametric SIM are those for the mean regression. See \cite{Powell1989}, \cite{Ichimura1993}, \cite{Hardle1993a}, see also \cite{Horowitz2009} for a recent review. In such models, the  index and the conditional mean given the index are unknown. SIM for quantile regression were considered recently, see \cite{Kong2012}. A more restrictive, but still of significant interest, class of models is obtained by imposing the single-index paradigm to the conditional distribution of response variable given a vector of covariates. In these cases the index and the conditional law of the response given the index are unknown. The famous Cox proportional hazard model, see \cite{Cox1972}, is a particular case of SIM for conditional laws. See \cite{Delecroix2003}, \cite{Hall2005a}, \cite{Chiang2012} for more general situations.

The large amount of interest for SIM could be explained by the fact that the single-index assumption is very often the first intermediate step from a parametric framework towards a  fully nonparametric paradigm. Then an important question is whether this dimension reduction compromise is good enough  to capture the relevant information contained in the covariate vector. A possible way to answer is to build a statistical test of the single-index assumption against general alternatives. Several tests of the goodness-of-fit of single-index mean regression models have proposed in the literature. See \cite{FanLi1996}, \cite{Xia2004}, \cite{Stute2005}, \cite{Chen2009}, \cite{Escanciano2010} and the references therein. The problem of testing SIM models for conditional distribution in full generality seems open.

In this paper we propose a new and quite simple  kernel smoothing-based approach for testing single-index assumptions. We focus on mean regression and conditional law models. The approach is inspired by the remark that, up to some error in covariates, the single-index assumption check could be interpreted as a test of significance in nonparametric regression. Next, the single-index assumption could be conveniently reformulated as an equivalent unconditional moment condition. Finally, a kernel based test statistic could be used to test the unconditional moment condition. The smoothing based goodness-of-fit test approach allows to make the error in covariates negligible and thus to obtain a pivotal asymptotic law under the null hypothesis. The covariate vector needs not have a density, discrete covariables are allowed. Only the index estimated under the SIM assumption is used in kernel smoothing and this fact mitigates the effect of high-dimensional covariates. Meanwhile the asymptotical critical values are given by the quantiles of the normal law. Irrespective of the fixed dimension of the covariate vector, the new test detects local alternatives approaching the null hypothesis slower than $n^{-1/2}h^{-1/4},$ where $h$ is the bandwidth used to build the test statistic and $n$ is the sample size.

The paper is organized as follows.
In Section \ref{sec1}, we recall general considerations on  single-index models.
In Section \ref{secGeneral}, we present a general approach of testing nonparametric significance
and in Section \ref{sec_sim_ch} we apply it to single-index hypotheses for mean regression as well as for conditional law.
In Section \ref{sec_emp_ev} we introduce a  wild bootstrap procedure to correct the asymptotic critical values with small samples and illustrate the performance of our test by an empirical study. Technical results and proofs are relegated to the appendix.


\section{Single-index models}\label{sec1}

Let $Y\in\mathbb{R}^d,$ $d\geq 1,$ denote the random response vector and let $X\in \mathbb{R}^p,$ $p\geq 1,$ be the random column vector of covariates. The data consists of independent copies of $(Y^\prime ,X^\prime)^\prime.$  For mean regression the single-index assumption means that there exists a column parameter vector $\beta_0\in\mathbb{R}^p $ such that
\begin{equation}\label{reg_sim}
\mathbb{E}[Y\mid X] = \mathbb{E}[Y\mid X^\prime \beta_0].
\end{equation}
Only the direction given by $\beta_0$ is identified, so that an additional identification condition accompanies the model assumption, as for instance $\|\beta_0\|=1$ and an arbitrary component is set positive, or an arbitrary component is set to 1. The scalar product $X^\prime \beta_0$  is the so-called index. The direction $\beta_0$ and the nonparametric univariate regression $\mathbb{E}[Y\mid X^\prime \beta_0]$ have to be estimated. See \cite{Hristache2001}, \cite{Delecroix2006}, \cite{Horowitz2009}, \cite{Xia2011} and the references therein for a panorama of the existing estimation procedures.

When applying the single-index paradigm to conditional laws of $Y$ given $X,$ one supposes \begin{equation}\label{law_sim}
Y\perp  X\mid X^\prime \beta_0.
\end{equation}
In this case the direction defined by  $\beta_0$ and the conditional law of the response $Y$ given the index $X^\prime \beta_0$ have to be estimated.
See  \cite{Delecroix2003}, \cite{Hall2005a} and \cite{Chiang2012} for the available estimation approaches.

There are several model check approaches for SIM for mean regressions.
\cite{Xia2004} use an empirical process-based statistic related to that of \cite{Stute1998}.
\cite{FanLi1996}
use a kernel smoothing-based
quadratic form to a wide range of situations,
including single-index.
Our test statistics are somehow close to that of \cite{FanLi1996}.
\cite{Chen2009} use an empirical likelihood test for multi-dimensional $Y$ in a
parametric or semiparametric modeling, the single-index mean regression is presented as a particular case but without getting into the details.

In this paper we propose an alternative model check approach that is able to detect any departure from the single-index assumption, both for mean regressions and conditional law models. It is inspired by a general approach for testing nonparametric significance that is presented in the following section.

\section{A general approach for testing nonparametric significance}
\label{secGeneral}

Let $(\mathcal{H}, \langle\cdot,\cdot\rangle_{\mathcal{H}})$ be a  Hilbert space. The examples we have in mind corresponds to  $\mathcal{H}=\mathbb{R}^d,$ for some $d\geq 1,$ or  $\mathcal{H}=L^2[0,1].$ Consider  $U\in\mathcal{H}$, $Z\in\mathbb{R}^{q}$ et $W\in\mathbb{R}^{r}$ and let   $(U_{i},Z_i,W_i)$, $1\leq i \leq n$ denote an independent sample of $U$, $Z$ and $W$.
Consider the problem of testing the equality
\begin{equation}
\mathbb{E}[U\mid Z,W] = 0 \qquad \text{p.s.}  \label{quang}
\end{equation}
against the nonparametric alternative $\mathbb{P}(\mathbb{E}[U\mid Z,W] = 0 ) <1.$
Several testing procedures against nonparametric alternatives, including the single-index assumptions check, lead to this type of problem.

Let us introduce some notation: for any real-valued, univariate or multivariate function $l$, let  $\mathcal{F}[l]$ denote the Fourier Transform of
$l$. Let $K$ be a multivariate kernel $\mathbb{R}^{q}$ such that
$\mathcal{F}[K]>0$
and let $\phi(s)=\exp(-\|s\|^2/2),$  $\forall s\in\mathbb{R}^r.$
The kernel $K$ could be a multiplicative  kernel with univariate kernels with positive Fourier Transform. Many univariate kernels have this property: gaussian, triangle, Student, logistic, etc.

Our approach is based on the following remark; see also \cite{Lavergne2014}. Let  $w (\cdot)>0$ be some weight function. For any  $h>0$, let
\begin{multline}
I(h) =  \mathbb{E}\left[\langle U_{1}, U_{2}\rangle_{\mathcal{H}}\; w(Z_{1})w(Z_{2})h^{-q}K((Z_{1}-Z_{2})/h)\phi(W_{1}-W_{2})\right]\nonumber \\
 =  \mathbb{E}\left[\langle U_{1}, U_{2}\rangle_{\mathcal{H}}\; w(Z_{1})w(Z_{2})\int_{\mathbb{R}^{q}}e^{2\pi iv^{\prime}(Z_{1}-Z_{2})}\mathcal{F}[K](vh)dv\int_{\mathbb{R}^r}e^{2\pi is^\prime(W_{1}-W_{2})
 }\mathcal{F}[\phi](s)ds\right]\nonumber \\
 =  \int_{\mathbb{R}^{q}}\int_{\mathbb{R}^{r}}\left\|\mathbb{E}\left[\mathbb{E}[U\mid Z,W]w(Z)e^{-i\{v^{\prime}Z +  s^\prime W \}}\right]\right\|_{\mathcal{H}}^{2}\mathcal{F}[K](vh)\mathcal{F}[\phi](s)dtds.
\end{multline}
Since  $\mathcal{F}[\phi], \mathcal{F}[K]>0,$ and $w(\cdot)>0$, the following equivalence holds true: $\forall h>0$,
\[
\mathbb{E}[U\mid Z,W]=0\;\; p.s.\;\;\Leftrightarrow\;\; I(h)=0.
\]
To check condition (\ref{quang}) the idea is to build a sample based approximation of $I(h),$ to suitably normalize it and to let $h$ to decrease to zero. A convenient choice of $w(\cdot)$ could avoid handling denominators close to zero.

In many situations the sample of the variable $U w (Z)$ is not observed and has to be estimated inside the model. Then, an estimate of  $I(h)$ is given by the $U-$statistic
\[
I_{n}(h) = \frac{1}{n(n-1)h^{q}}\sum\limits _{1\leq i\neq j\leq n}\left\langle \widehat {U_{i}w(Z_{i})}, \;\widehat{U_{j}w(Z_{j})} \right\rangle_{\mathcal{H}} K_{ij}(h)\;
\phi_{ij},
\]
where
$$
K_{ij}(h) = K((Z_i - Z_j)/h), \qquad \phi_{ij}  = \exp (-\|W_i - W_j\|^2/2).
$$
The variance of  $I_n(h)$ could be estimated by
$$
v^2_{n}(h) = \frac{2}{n^2(n-1)^2h^{2q}}\sum\limits _{1\leq i\neq j\leq n}\left\langle \widehat {U_{i}w(Z_{i})}, \;\widehat{U_{j}w(Z_{j})} \right\rangle^2_{\mathcal{H}} K^2_{ij}(h)\;
\phi^2_{ij}.
$$
Then the test statistic is
$$
T_n = \frac{I_n(h)}{v_{n}(h)}.
$$

Under mild technical conditions and provided that $h$ converges to zero at a suitable rate, $T_n$ converges in law to a standard normal distribution provided that condition (\ref{quang}) holds true. Hence, a one-sided test with standard normal critical values could be defined; see \cite{Lavergne2014}. One could also show $T_n$ tends to infinity in probability if $\mathbb{P}(\mathbb{E}[U\mid Z,W] = 0 ) <1.$
Making $h$ to decrease to zero at suitable rate allows to render negligible the effect of the errors
$\widehat {U_{i}w(Z_{i})}- {U_{i}w(Z_{i})}.$
On the other hand, the test detects  Pitman alternative hypotheses like
\begin{equation}\label{pitman}
H_{1n}:\ \mathbb{E}(U\mid Z,W)= r_{n}\delta(Z,W),\quad n\geq 1,\;
\end{equation}
as soon as $r_{n}^{2}nh^{q/2}\rightarrow\infty$.

\section{Single-index assumptions checks}\label{sec_sim_ch}

In this section we extend the approach described in section (\ref{secGeneral}) to test single-index assumptions like (\ref{reg_sim}) and (\ref{law_sim}). In this case,  with the notation from section \ref{secGeneral}, $$q=1, \;\; r=p-1, \;\;Z=Z(\beta)\;\; \text{ and  } \;\;W=W(\beta)$$ where, for $\beta\in \mathcal{B} \subset \mathbb{R}^p,$
$$
Z(\beta)=X^\prime \beta \qquad \text{et} \qquad W(\beta) =X^\prime \mathbf{A}\left(\beta\right)
$$
with $\mathbf{A}\left(\beta\right)$ a $p\times (p-1)$ matrix with real entries such that the $p\times p$ matrix
$\left(
\beta \; \mathbf{A}\left(\beta\right)\right)$
is orthogonal. The orthogonality is not necessary, invertibility suffices, but orthogonality is expected to lead to better finite sample properties for the tests.

An additional challenge will come from the fact that the sample of the covariates $Z$ and $W$ depend on estimator of the single-index direction $\beta_0.$ Again, the kernel smoothing and a suitable choice of $h$ allows to render this effect negligible and preserve a pivotal asymptotic law under the null hypothesis.

\subsection{Testing SIM for mean regression}\label{subsec_1}

To simplify the presentation, let us focus on the case of a univariate response, that is $d=1.$ At the end, it will be quite clear how the case $d>1$ could be handled. To restate the single-index condition (\ref{reg_sim}), let
$\mathcal{H} = \mathbb{R},$ $U w(Z) = U (\beta_0) w(Z;\beta_0)$ where
$$
U (\beta) w(Z;\beta) = \{ Y - \mathbb{E}[ Y \mid Z(\beta) ] \} f_{\beta} ( Z(\beta)).
$$
Here $f_\beta(\cdot)$ denotes the density of $X^\prime \beta$ that is supposed to exist, at least for some $\beta.$
Let
\begin{equation}\label{rrr}
\widehat {U_{i}w(Z_{i})} (\beta)=  \frac{1}{n-1}\sum\limits _{k\neq i} (Y_i - Y_k)\frac{1}{g}L_{ik}(\beta,g),
\end{equation}
where $L$ is a univariate kernel, $L_{ik}(\beta,g) = L((Z_i(\beta) - Z_k(\beta))/g)$ and $g$ is a bandwidth converging to zero at some suitable rate described in a following section.
Let $\hat \beta$ be some  estimator of the index direction and  consider
\begin{eqnarray*}
I_{n}^{\{m\}} (\hat{\beta})& = & \dfrac{1}{n\left(n-1\right)h}\sum_{1\leq i\neq j \leq n}\widehat {U_{i}w(Z_{i})} (\hat \beta)\widehat {U_{j}w(Z_{j})} (\hat \beta)K_{ij}(\hat{\beta},h)\phi( W_{i} (\hat{\beta})-W_{j} (\hat{\beta} ) ),
\end{eqnarray*}
where  $K_{ij}(\hat\beta,h) = K((Z_i(\hat\beta) - Z_j(\hat\beta))/h).$ The variance of
$I_{n}^{\{m\}} (\hat{\beta})$ could be estimated by
\begin{equation*}
\hat{\omega}_{n}^{\{m\}} (\hat{\beta})^{2}= \dfrac{2}{n^2\left(n-1\right)^2 h^2}\sum_{1\leq i\neq j \leq n}\!\!\left[\widehat {U_{i}w(Z_{i})}(\hat \beta)\widehat {U_{j}w(Z_{j})}(\hat \beta)\right]^2K_{ij}^2(\hat{\beta},h)\phi^2( W_{i}(\hat{\beta})-W_{j}(\hat{\beta})).
\end{equation*}
The test statistic is then
\begin{equation*}
T_{n}^{\{m\}}(\hat{\beta}) = \dfrac{I_{n}^{\{m\}} (\hat{\beta})}{\hat{\omega}_{n}^{\{m\}}(\hat{\beta})}.
\end{equation*}
Let us point out that only smoothing with the $X_i^\prime \hat \beta$'s is required in order to build this statistic.

In section \ref{sec_as_th} we show that whenever $ \hat{\beta} - \beta^* = O_{\mathbb{P}}(n^{-1/2}), $ for some  $\beta^*$ that could depend on $n,$
\begin{equation}\label{equivalence}
I_{n}^{\{m\}} (\hat{\beta}) - I_{n}^{\{m\}} (\beta^*)= o_{\mathbb{P}}(I_{n}^{\{m\}} (\beta^*)) \quad \text{ and } \quad \hat{\omega}_{n}^{\{m\}} (\hat{\beta}) - \hat{\omega}_{n}^{\{m\}} (\beta^*)=o_{\mathbb{P}}(\hat{\omega}_{n}^{\{m\}} (\beta^*)),
\end{equation}
provided  some mild technical conditions hold true. Under the null hypothesis (\ref{reg_sim}) one expects to have $\beta^* = \beta_0.$ Then $T_{n}^{\{m\}}(\hat{\beta})$ has an asymptotic  standard normal law under the single-index assumption as soon as $T_{n}^{\{m\}}({\beta}_0)$ is standard normal asymptotically distributed.
Sufficient conditions for guaranteeing the asymptotic normality of  $T_{n}^{\{m\}}({\beta}_0)$ when (\ref{reg_sim}) holds true  are provided in \cite{Lavergne2014}.

When the SIM  (\ref{reg_sim}) is wrong, even asymptotically,
in general a semiparametric estimator $\hat \beta$ converges at the rate $O_{\mathbb{P}}(n^{-1/2})$ to some  \emph{pseudo-true} value $\beta^*\in\mathcal{B}$ that depends on the estimation procedure; see \cite{Delecroix1999} for some general theoretical results. Then the asymptotic equivalence (\ref{equivalence}) and the results of \cite{Lavergne2014} imply that a test based on $T_{n}^{\{m\}}(\hat{\beta})$ would reject the null hypothesis with probability tending to 1, in just the way the test based on $T_{n}^{\{m\}}({\beta}^*)$  would do. The case of Pitman alternatives requires a longer investigation since the conclusion depends on the estimation method and the properties of the deviation from the null hypothesis. Such a detailed investigation is beyond our present  scope. Let us, however, briefly describe what would happen in the case where the index $\beta_0$ was estimated through a semiparametric least-squares procedure as introduced by \cite{Ichimura1993}.
Let $r_\beta (s) = \mathbb{E}[Y \mid  X^\prime \beta =s]$ and
$$
\nabla_\beta   \mathbb{E}(Y \mid  X^\prime \beta_0) = \left. \frac{\partial }{\partial \beta} \; r_\beta (X^\prime \beta) \right|_{\beta=\beta_0}.
$$
Let $\delta(X)$ satisfy  $\mathbb{E}[\delta(X) \mid  X^\prime \beta_0] =0 $ and $\mathbb{E}[\delta(X) \nabla_\beta   \mathbb{E}(Y \mid  X^\prime \beta_0) \tau (X)] =0 $ where $\tau(\cdot)$ is a trimming function required in theory to keep the denominators appearing in kernel smoothing away from zero. See, for instance,  \cite{Delecroix2006} for detailed discussion on the role of the trimming.
Consider the  sequence of alternatives
$$
\mathbb{E}(Y \mid X) = \mathbb{E}(Y\mid X^\prime \beta_0) + r_n \delta (X), \quad n\geq 1,
$$
with $r_n\rightarrow 0.$ Then it can be proved that
$ \hat{\beta} - \beta_0  = O_{\mathbb{P}}(n^{-1/2}),$ and hence
$T_{n}^{\{m\}}(\hat{\beta})$ allows to detect such local alternatives  as soon as
$r_{n}^{2}nh^{1/2}\rightarrow\infty$.

\subsection{Testing SIM for the conditional law}\label{subsec_2}
In order to test the single-index condition (\ref{law_sim}) for the conditional law of an univariate $Y$ given $X,$ let  $\mathcal{H} =L^2 [0,1]$ et
$$
U (t;\beta) w(Z;\beta) = \left\{ \mathbf{1}\{ \Phi (Y) \leq t\} - \mathbb{P}[ \Phi (Y)\leq  t\mid Z(\beta) ] \right\} f_{\beta} ( Z( \beta)),  \quad t\in[0,1], \; \beta\in\mathcal{B},
$$
where $\Phi$ is some distribution function on the real line, for instance a normal distribution function or the marginal distribution function of $Y$. In the latter case, in general the distribution is unknown but could be estimated by the empirical distribution function. The case of multivariate $Y$ could be also considered after obvious modifications and for the sake of simplicity will not be investigated herein.

Let
\begin{equation}\label{rrr2}
\widehat {U_{i}w(Z_{i})}( \beta) (t) =  \frac{1}{n-1}\sum\limits _{k\neq i} (
\mathbf{1}\{\Phi (Y_i) \leq t\} - \mathbf{1}\{\Phi (Y_k)\leq t\})\frac{1}{g}L_{ik}( \beta,g), \quad t\in[0,1].
\end{equation}
Let $\widetilde \beta$ be some estimator of $\beta_0$ and  consider
\begin{eqnarray*}
I_{n}^{\{l\}} (\widetilde{\beta})& = & \dfrac{1}{n\left(n-1\right)h}\sum_{1\leq i\neq j \leq n}\!\!\left\langle \widehat {U_{i}w(Z_{i})}(\widetilde \beta),\; \widehat {U_{j}w(Z_{j})} (\widetilde \beta)\right\rangle_{L^2}K_{ij}(\widetilde{\beta},h)\phi( W_{i} (\widetilde{\beta})-W_{j} (\widetilde{\beta} )),
\end{eqnarray*}
where  for any $u(\cdot)$ and $v(\cdot)$ squared integrable functions defined on the unit interval, $$\langle u,v \rangle_{L^2} = \int_0^1 u(t)v(t)dt.$$
The variance of
$I_{n}^{\{l\}} (\tilde{\beta})$ could be estimated by
\begin{equation}\label{om_hat}
\hat{\omega}_{n}^{\{l\}} (\tilde{\beta})^{2}= \dfrac{2}{n^2\left(n\!-\!1\right)^2 \! h^2}\sum_{1\leq i\neq j \leq n}\!\!\!\left\langle\widehat {U_{i}w(Z_{i})}(\widetilde \beta), \;\widehat {U_{j}w(Z_{j})}(\widetilde \beta) \right\rangle_{L^2}^2 \!\!K_{ij}^2(\tilde{\beta},h)\phi^2( W_{i}(\tilde{\beta})-W_{j}(\tilde{\beta})).
\end{equation}
The test statistic is then
\begin{equation*}
T_{n}^{\{l\}}(\tilde{\beta}) = \dfrac{I_{n}^{\{l\}} (\tilde{\beta})}{\hat{\omega}_{n}^{\{l\}} (\tilde{\beta})}.
\end{equation*}

In section \ref{sec_as_th} we show that, under suitable technical conditions, whenever $ \widetilde{\beta} - \beta^\sharp = O_{\mathbb{P}}(n^{-1/2}), $
\begin{equation}\label{equivalence2}
I_{n}^{\{l\}} (\widetilde{\beta}) - I_{n}^{\{l\}} (\beta^\sharp)= o_{\mathbb{P}}(I_{n}^{\{l\}} (\beta^\sharp)) \quad \text{ and } \quad \hat{\omega}_{n}^{\{l\}} (\widetilde{\beta}) - \hat{\omega}_{n}^{\{l\}} (\beta^\sharp)=o_{\mathbb{P}}( \hat{\omega}_{n}^{\{l\}} (\beta^\sharp)).
\end{equation}
Under the null hypothesis (\ref{law_sim}) one expects to have $\beta^\sharp = \beta_0.$
Then the asymptotic normality of $T_{n}^{\{l\}}({\beta}_0),$ proved in  Proposition \ref{sam_1mai} below,
implies that the asymptotic one-sided test based on $T_{n}^{\{l\}}(\tilde{\beta})$ has standard normal critical values.

If the single-index assumption fails and the alternative is fixed, like in the case of mean regression, one expects $\widetilde \beta - \beta^*=O_{\mathbb{P}}(n^{-1/2})$ for some  \emph{pseudo-true} value $\beta^*\in\mathcal{B}$ that depends on the estimation procedure. Then $T_{n}^{\{l\}}(\tilde{\beta})$ would detect the alternative with probability tending to 1. Concerning the case of local alternatives, let $\delta (X,t)$ and $r_n\rightarrow 0$ such that
$$
\mathbb{P} [\Phi (Y) \leq t\mid X ] =  \mathbb{P} [\Phi (Y)\leq t \mid X^\prime \beta_0 ] + r_n \delta(X,t),\qquad t\in[0,1],
$$
is a conditional distribution function. Suitable orthogonality conditions for the function $\delta (X,t)$ would yield $\widetilde \beta - \beta_0=O_{\mathbb{P}}(n^{-1/2})$ and hence
$T_{n}^{\{l\}}(\widetilde{\beta})$ allows to detect such local alternatives  as soon as
$r_{n}^{2}nh^{1/2}\rightarrow\infty$.

\subsection{Asymptotic results}\label{sec_as_th}

In this section we formally state the results that guarantee the asymptotic equivalences
(\ref{equivalence}) and (\ref{equivalence2}).
Let $\widehat {U_{i}w(Z_{i})}(\beta)$ be defined as in (\ref{rrr}) or (\ref{rrr2}). Let $I_{n} ({\beta})$ (resp. $\hat{\omega}_{n} ({\beta})^{2}$) denote any of $I^{\{m\}}_{n} ({\beta})$ or $I^{\{l\}}_{n} ({\beta})$ (resp. $\hat{\omega}_{n}^{\{m\}} ({\beta})^{2}$ or $\hat{\omega}_{n}^{\{l\}} ({\beta})^{2}$).

\begin{proposition}\label{as_equiv}
Suppose  the conditions in Assumption \ref{ass_app} are met. If $\beta_n$ is an estimator such that
$
{\beta_n} - \bar\beta = O_{\mathbb{P}}(n^{-1/2}), $
then
$$
I_{n}({\beta_n}) - I_{n} (\bar \beta)= o_{\mathbb{P}}(I_{n}(\bar\beta)) \quad \text{ and } \quad \hat{\omega}_{n} ({\beta_n}) - \hat{\omega}_{n} (\bar\beta)=o_{\mathbb{P}}(\hat{\omega}_{n} (\bar\beta)).
$$
\end{proposition}

As mentioned above, the asymptotic behavior of $I_{n} (\bar \beta)$ in the case of mean regression was investigated by \cite{Lavergne2014}. The case where $ U_{i}w(Z_{i})(\beta)$ is a stochastic process seems less explored and is hence considered in the following proposition. Let $\hat \omega_n \left(\beta_{0}\right)$
be a variance estimator defined as in equation (\ref{om_hat}) with $\widetilde \beta$ replaced by $\beta_0.$

\begin{proposition}\label{sam_1mai}
Suppose  the conditions in Assumption \ref{ass_app} are met and the null hypothesis
(\ref{law_sim}) holds true. Then
$nh^{1/2}I_{n}^{\{l\}}\left(\beta_{0}\right)/\hat \omega_n^{\{l\}}\left(\beta_{0}\right)\to\mathcal{N}\left(0,1\right)$ in law
under $H_{0}$, and
\begin{multline*}
\hat \omega^{\{l\}}_n\left(\beta_{0}\right)\to \omega^{2}\left(\beta_{0}\right)  =  2\int K^{2}\left(u\right)du\times\intop\intop\Gamma^{2}\left(s,t\right)ds\, dt\\
  \times\mathbb{E}\left[\intop f_{\beta_{0}}^{4}\left(z\right)\phi^{2}\left(W_{1}\left(\beta_{0}\right)-W_{2}\left(\beta_{0}\right)\right)\pi_{\beta_0}\left(z\mid W_{1}\left(\beta_{0}\right)\right)\pi_{\beta_0}\left(z\mid W_{2}\left(\beta_{0}\right)\right)dz\right],
\end{multline*}
where $\pi_{\beta_0}(\cdot,w)$ is the conditional density of $Z(\beta_0)$ knowing that $W(\beta_0)=w,$
 and
\begin{eqnarray*}
\Gamma\left(s,t\right) & = & \mathbb{E}\left[\epsilon\left(s\right)\epsilon\left(t\right)\right] \qquad t,s\in[0,1],
\end{eqnarray*}
and  $\epsilon\left(t\right) = \mathbf{1}\{\Phi (Y)\leq t\}- \mathbb{P} [\Phi (Y)\leq t \mid X^\prime \beta_0 ]. $

\end{proposition}


\section{Empirical evidence}\label{sec_emp_ev}

For conditional mean, we simulate our data using the following model
\begin{eqnarray}
Y_{i} & = & X_{i}^{\prime}\beta+4\exp \{ -\left(X_{i}^\prime \beta\right)^{2} \} +\delta\sqrt{X_{i}^{\prime} X_{i}}+ \sigma \varepsilon_{i}
\label{eq:Mod1}\end{eqnarray}
where $X_{i}=(X_{i1},\cdots,X_{ip})^\prime $ follows a standard normal $p$-variate law, $\beta_0=\left(1,1,0,\dots,0\right)^\prime$ and $\sigma=0.3.$
For $\varepsilon_{i}$, we consider two cases: a standard univariate normal law independent of the $X_i$'s
and a centered log-normal heteroscedastic setup
$$
\varepsilon_{i}=\left(\log\mathcal{N}(0,1)-\sqrt{\mbox{e}}\right)\times\sqrt{(1+X_{i2}^2)/2}.$$
The model (\ref{eq:Mod1}) was proposed by \cite{Xia2004} and investigated only in the case of a homoscedastic noise.

To estimate the parameter $\beta,$ we consider the approach of \cite{Delecroix2006}, that is
\begin{equation}\label{est_paramj}
\tilde{\beta}=\arg\min_{\beta: \beta_{1}>0}\sum_{i=1}^n\left(Y_i - \dfrac{\sum_{k\neq i}Y_{k}\tilde{L}_{ik}\left(\beta\right)}{\sum_{k\neq i}\tilde{L}_{ik}\left(\beta \right)}\right)^2,
\end{equation}
where
$$
\tilde{L}_{ik}(\beta)=L\left((\widetilde{X}_i-\widetilde{X}_k)^\prime \beta\right), \quad \widetilde{X}_i= \frac{X_i}{ \sqrt{n^{-1}\sum_{k=1}^n (X_k-\overline{X})^2}}\;\;\text{  and }\;\; \overline{X}=n^{-1}\sum_{k=1}^n X_k.$$
Then the estimator is defined as $\hat\beta = \tilde{\beta}  /\| \tilde{\beta} \|$ and the bandwidth $g$ is equal to
$\| \tilde{\beta} \|^{-1}.$

To improve the asymptotic critical values with small samples, we propose the following bootstrap procedure:
\begin{description}
%
%
\item [{(i)}] Define
\[
\hat{m}_{i}=\dfrac{\sum_{k\neq i}Y_{k}\tilde{L}_{ik}(\tilde{\beta})}{\sum_{k\neq i}\tilde{L}_{ik}(\tilde{\beta})}.
\]
\item [{(ii)}] For $b\in\left\{ 1,\dots,\, B \right\} $

\begin{description}
\item [{(a)}] let $Y_{i}^{*,\, b}=\hat{m}_{i}+\eta_{i}\left(Y_{i}-\hat{m}_{i}\right),$ where the $\eta_i$s are independent variables with the
 two-point distribution
\[
\mathbb{P} [ \eta_{i} = (1-\sqrt{5})/2 ]
  =
(5+\sqrt{5})/10
  \; , \;
\mathbb{P} [ \eta_{i} = (1+\sqrt{5})/2 ]
  =
(5-\sqrt{5})/10
.
\]
\item [{(b)}] define
\[
\tilde{\beta}^{*,b}=\arg\min_{\beta: \beta_{1}>0}\sum_{i=1}^n\left(Y^*_i - \dfrac{\sum_{k\neq i}Y^*_{k}\tilde{L}_{ik}\left(\beta\right)}{\sum_{k\neq i}\tilde{L}_{ik}\left(\beta \right)}\right)^2
\]
and $\hat\beta^{*, b}= \tilde{\beta}^{*, b}  /\| \tilde{\beta} ^{*, b}\|$ and $g^{*, b}= \| \tilde{\beta}^{*, b} \|^{-1}.$
\end{description}
\item [{(iii)}]
Define $T_{n}^{\{m\} *, b}$ as $T_{n}^{\{m\}}$ where the $Y_{i}$s
are replaced by the $Y_{i}^{*,\, b}$s, $\hat{\beta}$ by $\hat{\beta}^{*,b},$ and the bandwidth  $ g$ by $g^{*, b}.$ The bandwidth $h$ does not change. Repeat Step (iii) $B$  times. Compute the empirical quantiles of $T_{n}^{\{m\} *, b}$ using the $B$ bootstrap values.
\end{description}

\medskip

In our experiments  the bootstrap correction is used with $B=499$ bootstrap samples.
The level is fixed as $\alpha=10\%$.
We considered  $L\left(\cdot\right)=K\left(\cdot\right)$ and equal to the standard gaussian density. With this choice no numerical problem occurred due to denominators too close to zero and therefore we did not consider any trimming in equation (\ref{est_paramj}) and its bootstrap version.

First, we investigated the influence of the bandwidth $h$ on the level. Several bandwidths were considered, that
is $h=c\times n^{-2/9}$ with $c\in\{ 2^{k/2}: k=-2,-1,0,1,2\} $.
The results on empirical
rejection rates for the model defined in equation (\ref{eq:Mod1}) with   $\delta=0$ (that is on the null hypothesis) and $n=100$ are presented in Figure \ref{fig:SIMlevel}.
The results are based on $500$ replications, with homoscedastic noise and $p=2,$ $p=4$, and with heteroscedastic log-normal noise and $p=4$.
The normal critical values are quite inaccurate, while the bootstrap correction seems to overreject slightly, particularly
for a large bandwidth $h$. For the third case with heteroscedastic  noise, the test rejects too often. However, for larger sample
sizes, this drawback is mitigated, as could be seen from the fourth plot in Figure \ref{fig:SIMlevel} where we considered the heteroscedastic noise with $p=4$ and $n=200.$

Next, we studied the behavior of our statistic under the null hypothesis
(500 replications) and several alternatives (250 replications) defined
by some positive value of $\delta$.
We only considered the statistics
with bandwidth factor $c=1$ and compared it to the statistic introduced by \cite{Xia2004}.
The results are presented in Figure
(\ref{fig:SIMpower}).
\cite{Xia2004}'s test performs better for $p=2,$ while our test shows better performance for $p=4$. It appears  that the greater $p$ is,
the more advantageous it will be to use our test statistic.

For conditional law, we simulate our data using the following mixture model
\begin{eqnarray}
Y & = & \left(1-\delta\right)\mathcal{N}\left(X^{\prime}\beta,\,0.09\right)+\delta\mathcal{N}\left(\left\Vert X\right\Vert ,\,0.09\right)
\label{eq:Mod1law}\end{eqnarray}
where $X_{i}=(X_{i1},X_{i2})^\prime $ has a standard normal bivariate law and $\beta_0=\left(1,1\right)^\prime/\sqrt{2}$.
We apply the test statistic $I^{\{ l \}}_n$ based on the
quantities $\widehat {U_{i}w(Z_{i})}( \beta) (t)$ introduced in (\ref{rrr2}). Here  the events $\{\Phi (Y_i) \leq  t \}$ are defined  with $\Phi(\cdot)$ equal to the empirical distribution function of the $Y_i$'s.
In this case an event $\{\Phi (Y_i) \leq  t \}$ is determined by the rank of $Y_i$ in the sample of the response variable.
To estimate the index parameter $\beta$ we use the approach of \cite{Delecroix2003}
that we adapt to our particular choice of $\Phi(\cdot)$.
More precisely, let
\[
\left(\tilde{\beta},\tilde g_Y \right)=\arg\min_{(\beta,g_Y):\beta_1>0 }\sum_{i=1}^{n}\log\dfrac{\sum_{j\neq i}g_Y^{-1}L\left(\left(R_{i}-R_{j}\right)/(n g_Y) \right)\tilde{L}_{ij}(\beta)}{\sum_{j\neq i}\tilde{L}_{ij}(\beta)},
\]
where $R_i\in\{1,\cdots,n\}$ is the index of $Y_i$ in the order statistics $\{Y_{(1)},\cdots,Y_{(n)}\}, $ that is  $Y_{(R_i)}=Y_{i},$ $1\leq i\leq n$.
The aim is to estimate $\beta$ and $g$ simultaneously,
using again $\hat\beta = \tilde{\beta}  /\| \tilde{\beta} \|$ and the bandwidth $g=\| \tilde{\beta} \|^{-1}.$

For this test statistic, the bootstrap procedure considered is:
\begin{description}
\item [{(i)}] For $b\in\left\{ 1,\dots,\, B \right\} $
let $\widehat {U_{i}w(Z_{i})}^{*}( \beta) (t) = \eta_i \times \widehat {U_{i}w(Z_{i})}( \beta) (t)$
where the $\eta_i$'s are independent variables with the two-point distribution defined above.
\item [{(ii)}]
Define $T_{n}^{\{l\} *, b}$ as $T_{n}^{\{l\}}$ where the $\widehat {U_{i}w(Z_{i})}( \beta) (t)$'s
are replaced by the $\widehat {U_{i}w(Z_{i})}^{*}( \beta) (t)$'s. Repeat Step (i) $B$  times. Compute the empirical quantiles of $T_{n}^{\{l\} *, b}.$
\end{description}
We study the influence of bandwidth $h$ on empirical rejection under $H_0$ on the left part of Figure \ref{fig:SIMlaw},
where $h=c\times n^{-2/9}$ with $c\in\{ 2^{k/2}: k=-2,-1,0,1,2\}$, with $1000$ replications and $199$ bootstrap steps.
Because $\beta$ is not reestimated in the bootstrap procedure, we do not correct the estimation bias and the two rejection rate are very similar.
However, they are not far from the theoretical level.

We also investigate the empirical rejection rate for different values of the mixture proportion $\delta$ in the model (\ref{eq:Mod1law}).
The results are presented in the right panel of Figure \ref{fig:SIMlaw}.
We used $1000$ replications for $\delta=0$, $500$ replications otherwise, and $199$ bootstrap steps.
For $\delta$ from $0.1$ to $0.2$, the empirical rejection rate decreases, but it resumes its rise after $\delta=0.3$.
On the basis of the simulation results, we could explain this through the estimate of the variance of $I_n^{(l)}(\tilde\beta).$ A small deviation $\delta>0$ in the model (\ref{eq:Mod1law}) induces less variance for $Y$ and for the estimator of $\beta_0$. As a consequence, the tail of the estimator of the variance of $I_n^{(l)}(\tilde\beta)$ is lighter and this produces more power in the case $\delta=0.1.$ When the deviation $\delta$ slightly increases beyond $\delta=0.1$, the variance of $I_n^{(l)}(\tilde\beta)$ becomes too important and locally we observe a loss of power. When $\delta$ increases more, the increase of the variance of $I_n^{(l)}(\tilde\beta)$ is dominated by the increase $I_n^{(l)}(\tilde\beta)$ and the test has more power.

\begin{figure}
\includegraphics[scale=0.6,angle=270]{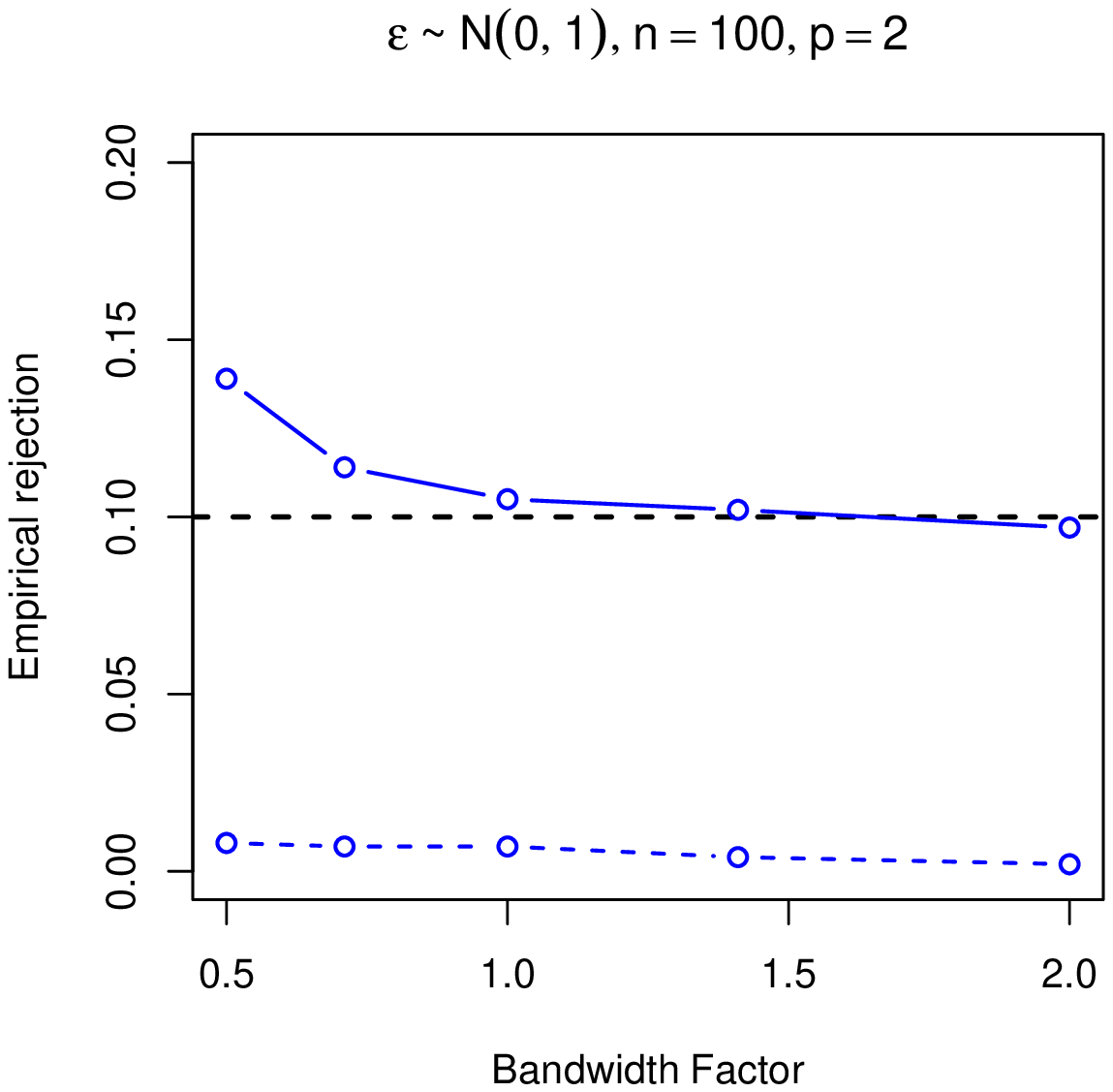}
\includegraphics[scale=0.6,angle=270]{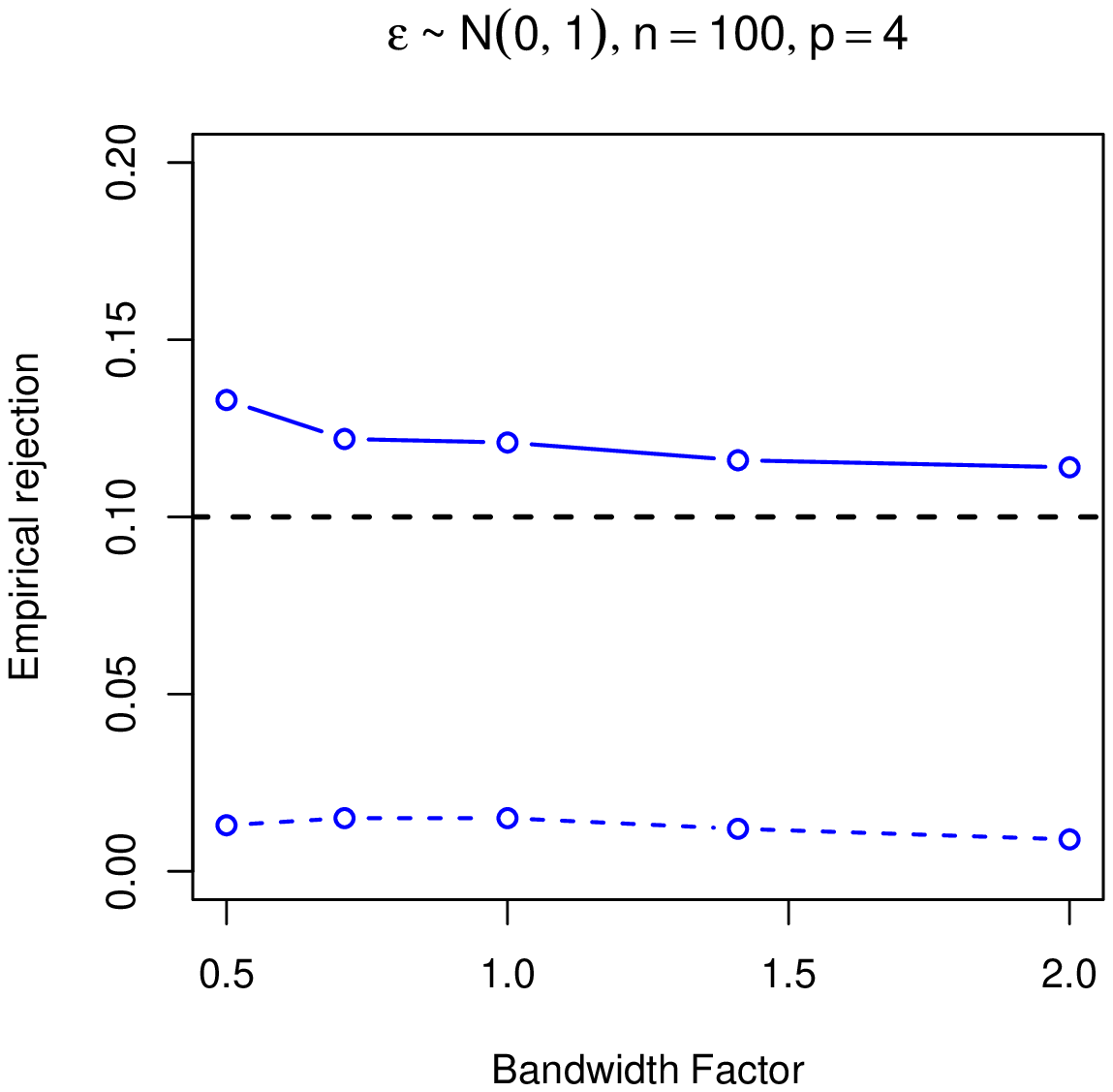}
\includegraphics[scale=0.6,angle=270]{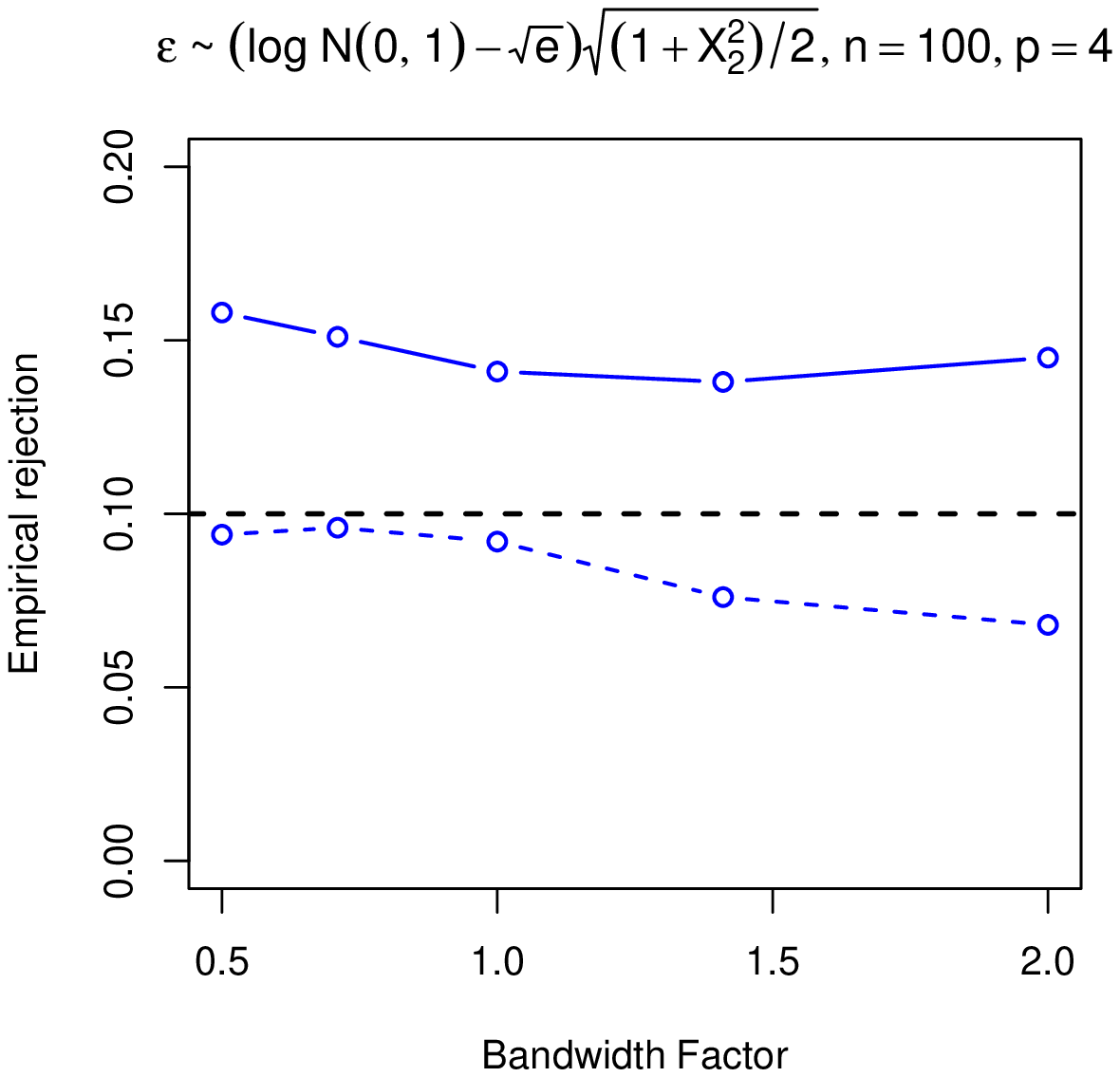}
\includegraphics[scale=0.6,angle=270]{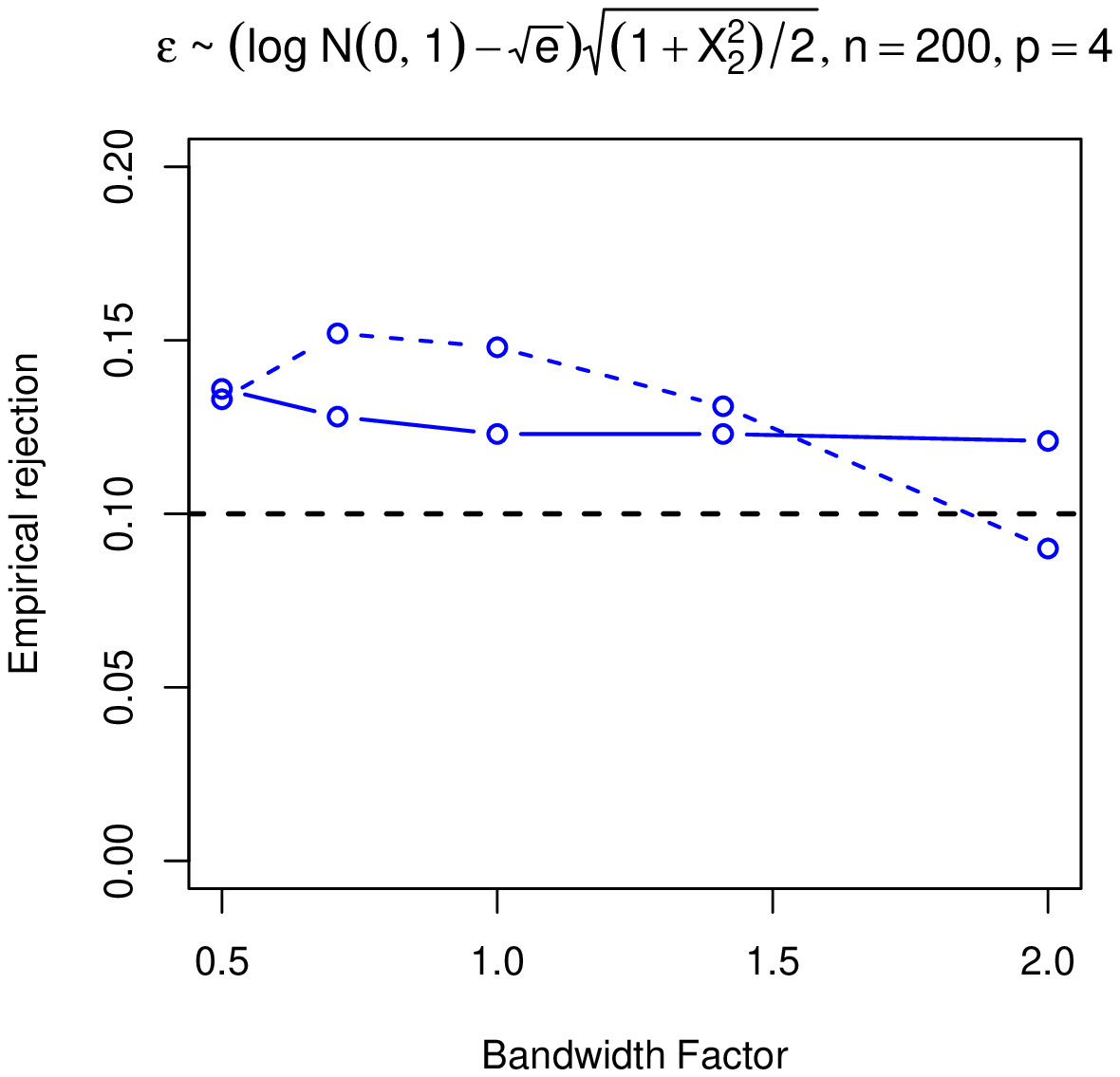}
\includegraphics[scale=0.6,angle=270]{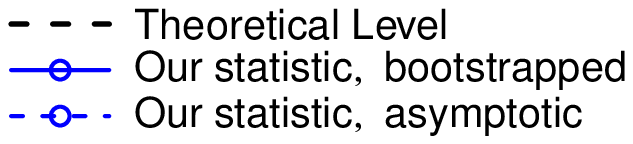}
 \vspace{-3.5cm}
\caption{Empirical rejections under $H_{0}$ as a function of the bandwidth
\label{fig:SIMlevel}}
\end{figure}
\begin{figure}
\includegraphics[scale=0.6,angle=270]{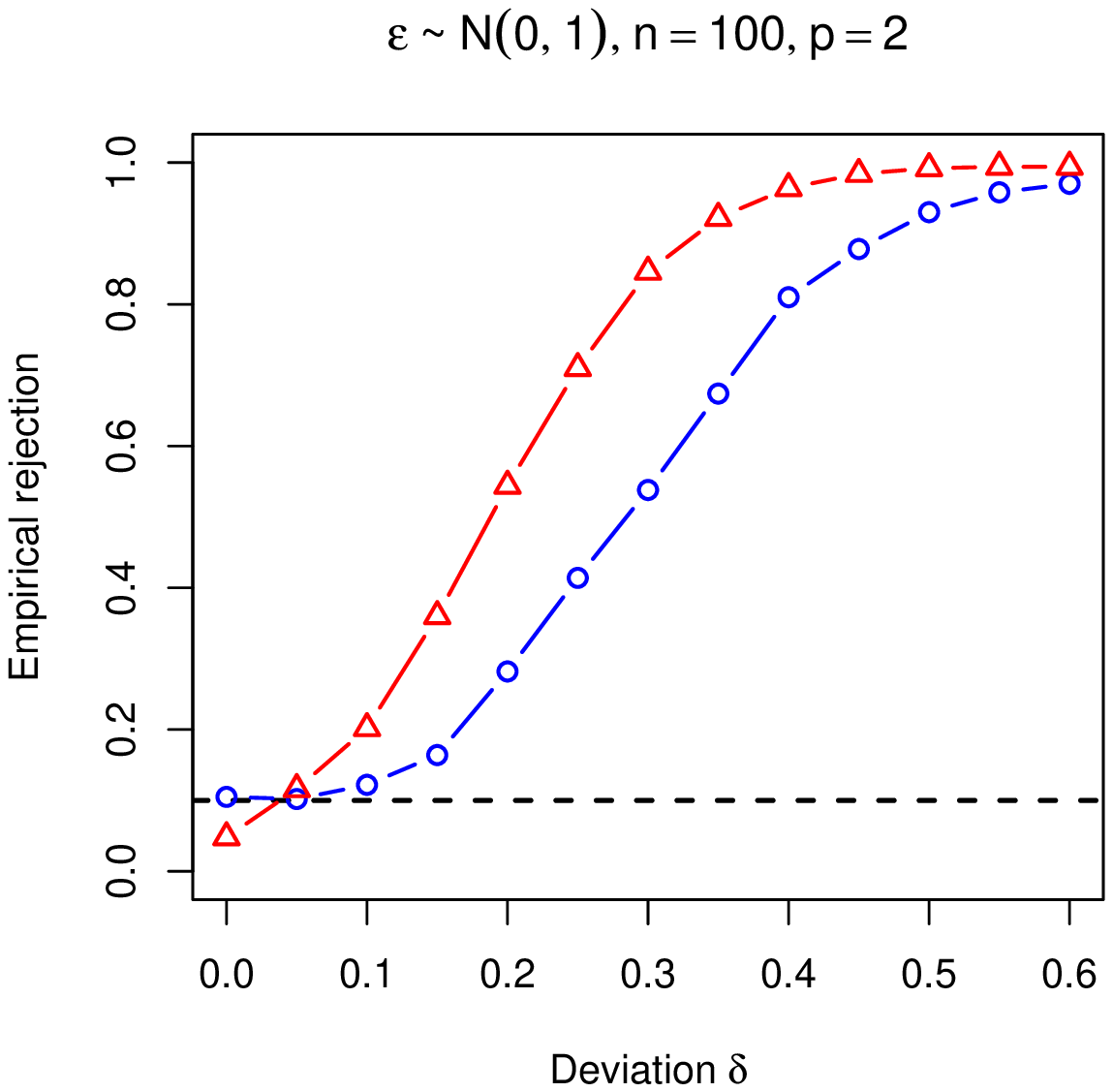}
\includegraphics[scale=0.6,angle=270]{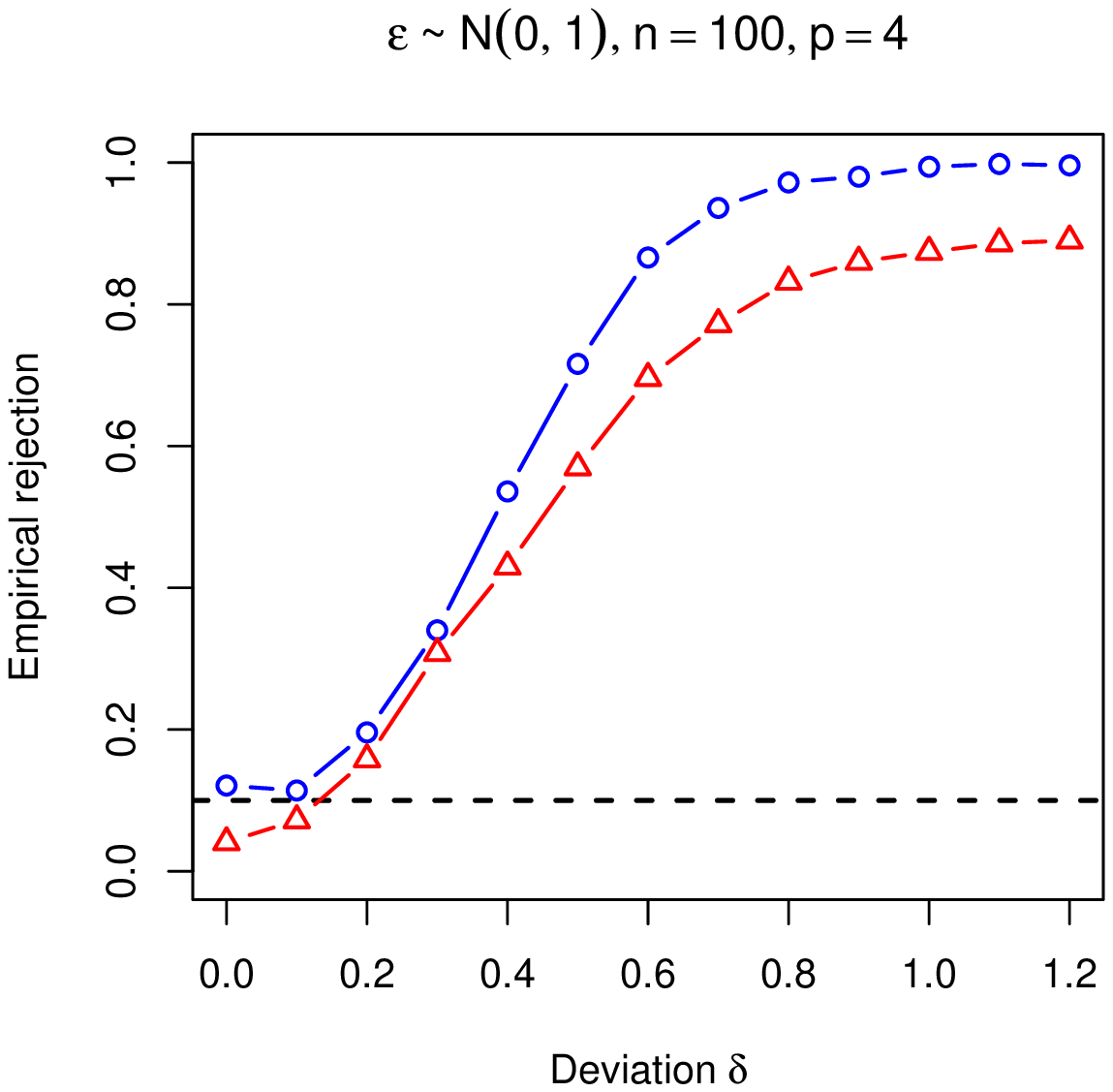}
\includegraphics[scale=0.6,angle=270]{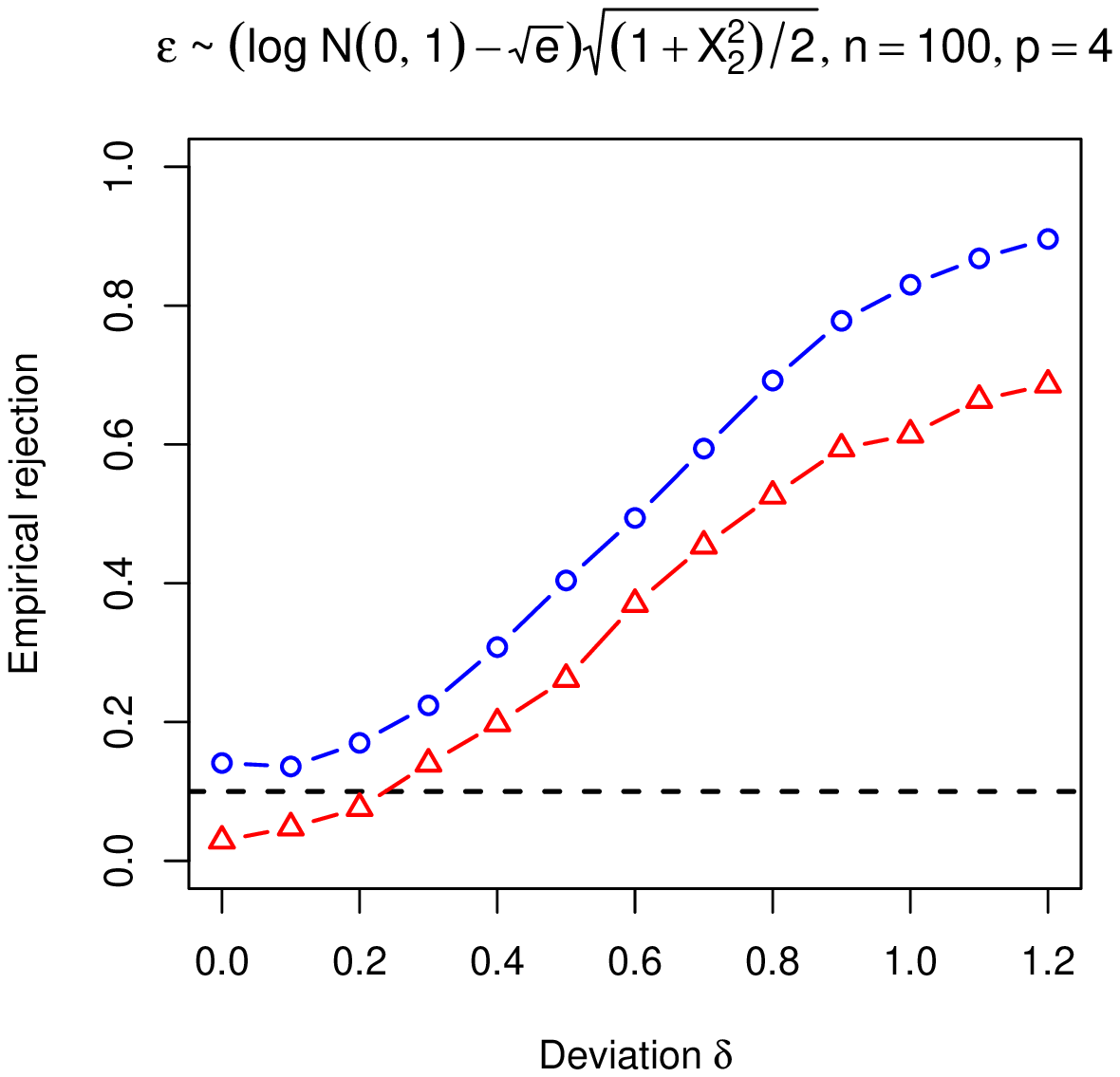}
\includegraphics[scale=0.6,angle=270]{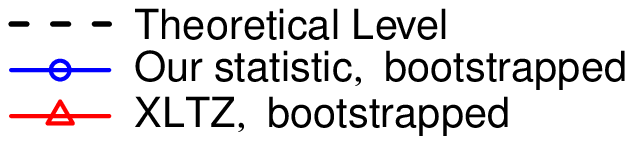}
\caption{Power curves for model (\ref{eq:Mod1}), $n=100$\label{fig:SIMpower}}
\end{figure}

\begin{figure}
\includegraphics[scale=0.6,angle=270]{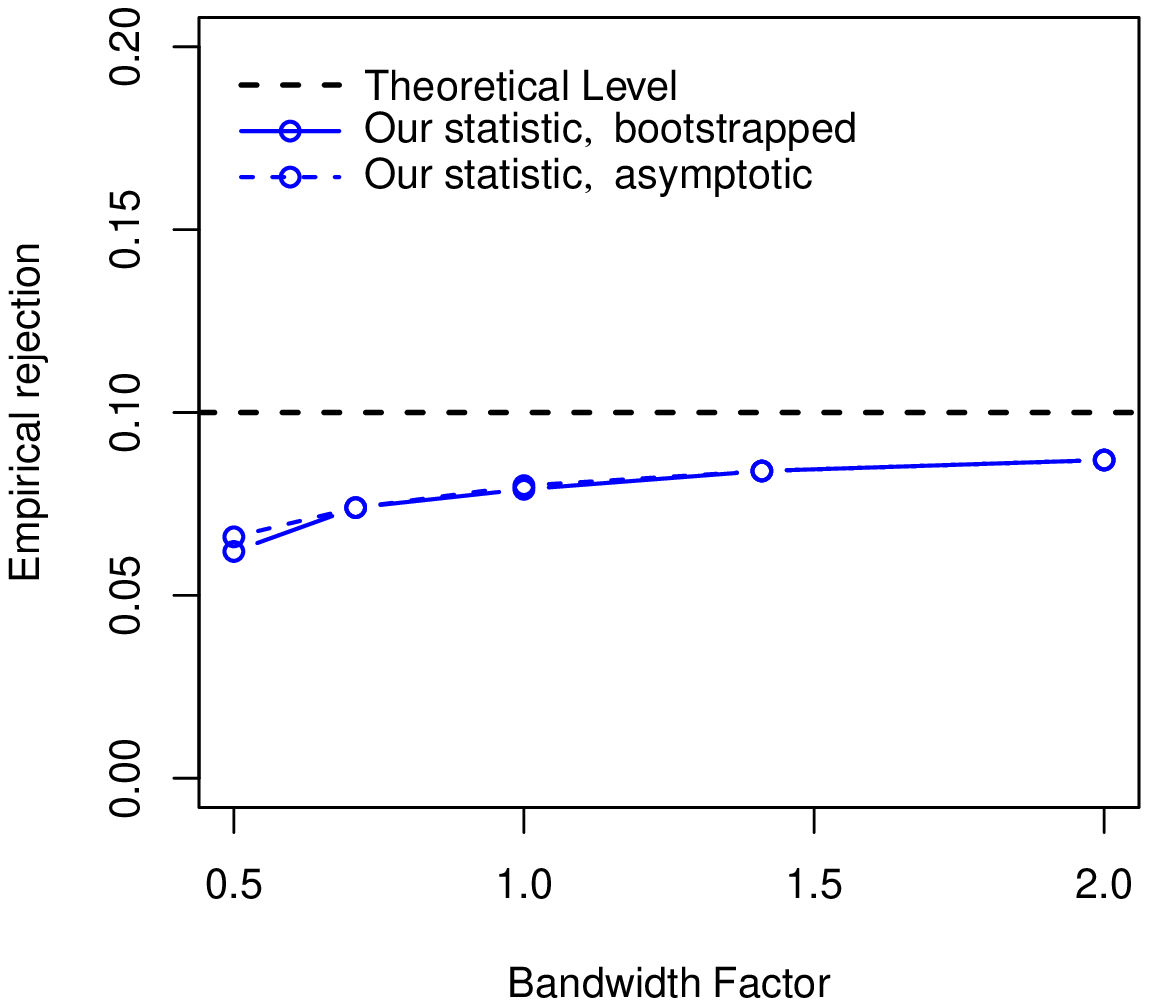}
\includegraphics[scale=0.6,angle=270]{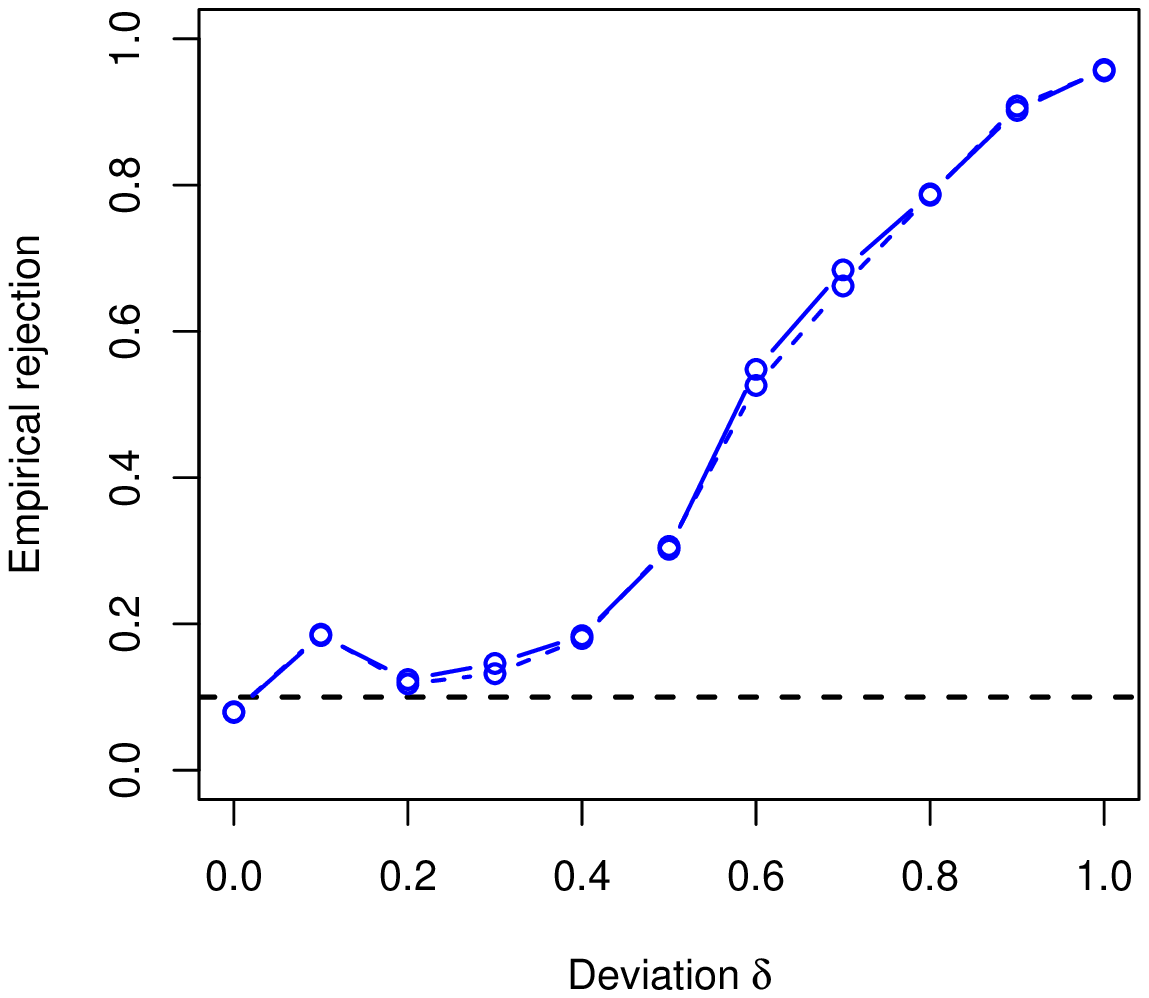}
\caption{Empirical rejections under $H_{0}$ and $H_{1}$ for conditional law, $n=200$.
On the left part, $h=c\times n^{-2/9}$ with varying $c$.
On the right part, $Y = \left(1-\delta\right)\mathcal{N}\left(X^{\prime}\beta,\,0.09\right)+\delta\mathcal{N}\left(\left\Vert X\right\Vert ,\,0.09\right)$.
\label{fig:SIMlaw}
}
\end{figure}

\section{Conclusions and furthers extensions}

We have constructed new smoothing-based test procedures for SIM hypotheses for mean regression
and for conditional law.
Smoothing is only used on the estimated index,
and the corresponding test statistics are asymptotically standard normal.
A quite effective wild bootstrap procedure allows to correct the critical values
with small samples. For simplicity we focused on univariate responses but, with obvious adjustments, our approach also applies to the case of multivariate responses. 
See \cite{Picone2000} and \cite{Chen2009} for more general situations
with multivariate responses where our test methodology applies.
Moreover, our statistics directly generalize to test multiple index against fully nonparametric alternatives.
It suffice to consider the general methodology presented in section  \ref{secGeneral}
with $q$ equal to the number of indices.
Some other possible extensions that would require additional, though quite straightforward,
investigation are the goodness-of-fit checks of index quantile regressions,
see \cite{Kong2012},
and the functional index models, see \cite{Chen2011}.
Such extensions are left for future work.


\clearpage

\bibliographystyle{ecta}
\bibliography{ManuscritAbre}

\newpage
\setcounter{equation}{0}
\section{Appendix 1: assumptions and proofs}

Let $\mathcal{H}$ be the real line or the Hilbert space of squared integrable functions defined on $[0,1].$ Let $\langle \cdot,\cdot\rangle_{\mathcal{H}}$ and $\| \cdot\|_{\mathcal{H}}$ denote the associated inner product and norm.
For an observation $(Y_i,X_i^\prime)^\prime,$ $Y_i\in\mathbb{R}$ and $X_i\in\mathbb{R}^p,$ let  $Y_i(t) \equiv Y_i$ or $Y_i(t) =\mathbf{1}\{Y_i\leq \Phi^{-1}(t)\} ,$ and for any $\beta$  in the parameter set $\mathcal{B}\subset\mathbb{R}^p,$ let $r_i(t;{\beta}) = \mathbb{E}[Y_i(t) \mid Z_i({\beta})],$ $t\in[0,1].$ Thus, $Y_i(\cdot)$ is an element of $\mathcal{H}.$

Let $(\epsilon_i(\cdot),X_i^\prime)^\prime,$ $1\leq i \leq n,$ be random variables such that
$\epsilon_i(\cdot)\in\mathcal{H} $ and $X_i\in \mathbb{R}^p .$
Let $\bar\beta$ be some element in the parameter set $\mathcal{B}.$ Consider $r_i(t;\bar{\beta})$  that depends only on $Z_i(\bar{\beta})=X_i ^\prime \bar\beta$ and $\delta(X_i,t)$ be such that $\mathbb{E}[\delta(X_i,t)\mid
Z_i(\bar{\beta})]=0,$ $t\in[0,1].$ Define
$$
Y_{ni}(t) = r_i(t;\bar{\beta}) + r_n \delta(X_i,t) + \epsilon_i (t) ,\qquad t\in[0,1], 1\leq i\leq n,
$$
where $r_n,$ $n\geq 1,$ is some bounded sequence of real numbers. In particular that means $\mathbb{E} [Y_{ni}(\cdot) \mid
Z_i(\bar{\beta})]=r_i(\cdot;\bar{\beta}).$ The case of a null sequence $(r_n)$ corresponds to the null hypothesis, while a sequence tending to zero corresponds to Pitman alternatives.

\begin{assumption}\label{ass_app}
a) The random variables $(\epsilon_i(\cdot),X_i^\prime)^\prime,$ $1\leq i \leq n,$ are independent copies of $\epsilon(\cdot)\in\mathcal{H} $ and $X\in \mathbb{R}^p .$
Moreover, $X^\prime \bar{\beta}$ admits a bounded density $f_{\bar{\beta}}.$

b) $\mathbb{E}[\exp(\rho\|X_i\|)]<\infty$ for some $\rho>0$ and
$\mathbb{E}[\sup_t |r_i(t;\bar\beta) + \epsilon_i(t) |^a ]<\infty$ for some $a>8.$
Moreover,
$\mathbb{E} (\|\epsilon_i (\cdot)\|^2_{\mathcal{H}}  \mid X_i)$ is bounded.

c) For any $t\in[0,1],$ the map $v \mapsto \mathbb{E} [Y_{ni}(t) \mid
Z_i(\bar{\beta})=v]$ is twice differentiable. The second derivative $r_i^{\prime\prime}(\cdot;\bar{\beta})$ is uniformly Lipschitz (that is the Lipschitz constant independent of $t$) and uniformly bounded, while the first derivative
satisfies $\mathbb{E}[\|r_i^{\prime}(\cdot;\bar{\beta})\|_{\mathcal{H}}^4]<\infty.$

d) The function $f_{\bar\beta}(\cdot)$ is uniformly Lipschitz.

e) The function $\delta(\cdot,\cdot)$ is bounded.

f) The kernels $K$ and $L$ are symmetric integrable functions, differentiable except at most a finite set of points and $L^\prime$ is Lipschitz continuous. Moreover,  $\int_{\mathbb{R}}|L (t)| dt = \int |K (t)| dt =1$ and $\int_{\mathbb{R}}(|L^\prime (t)| + |K^\prime (t)|) dt <\infty.$ The map $v\mapsto |L^\prime(v) | /v$ is bounded in a neighborhood of the origin, $v^2 K(v)\rightarrow 0$ if $v\rightarrow \infty,$ and  $\int v^2 \{|L(v)| + |K(v)| \}dv< \infty.$ Moreover, the Fourier Transform $\mathcal{F}[K]$ is positive on the real line.

g) The bandwidths satisfy the conditions
$g,h\rightarrow 0$,  $h/g^2\rightarrow 0,$ $nh^{1/2}g^{4}\rightarrow 0,$ $r_n^2nh^{1/2}\rightarrow \infty.$ Moreover, $g=n^{-\gamma}$ with $\gamma\in(1/5, 1/4)$ and thus $nh^2\rightarrow \infty.$

\end{assumption}

\quad

\begin{proof}[Proof Proposition \ref{as_equiv}]
First let us remark that for any $(a_n),$ a sequence divergent to infinity,
\begin{equation}\label{bn_an}
\mathbb{P}\left(\max_{1 \leq i \leq n} \|X_i\| > a_n \ln n\right) \rightarrow 0 \quad \text{and} \quad \mathbb{P}\left(\|\beta_n -\bar \beta\|> a_n n^{-1/2}\right) \rightarrow 0.
\end{equation}
Moreover, at least for $\beta $ in a fixed but small enough neighborhood of $\bar\beta$, the matrix $\mathbf{A}\left(\beta\right)$ could be built such that the norm of each of the $p-1$ columns of $\mathbf{A}\left(\beta\right) - \mathbf{A}\left(\bar \beta\right)$ is  bounded by $c \|\beta_n -\bar \beta\|$ with $c$ a constant independent of $\beta$. Indeed,  one could consider $p-1$ independent vectors which completed by any
$\beta$ close to $\bar\beta$ form a basis. Then one could use the Gram-Schmidt procedure to orthonormalize the basis. By construction, the norm of any columns of  $\mathbf{A}\left(\beta\right) - \mathbf{A}\left(\bar \beta\right)$ is bounded by $ c \|\beta_n -\bar \beta\|$ for some $c$ depending only on the initial $p-1$ independent vectors. All these facts show that we can reduce the parameter set to $\mathcal{B}_n,$ $n\geq 1,$ a sequence of balls centered at $\bar\beta$ of radius converging to zero. Consider the set of elementary events \begin{equation}\label{eventE}
\mathcal{E}_n = \left\{ \max_{1 \leq i \leq n} \sup_{\beta\in\mathcal{B}_n} [\|Z_i(\beta) - Z_i(\bar \beta)\| +
 \|W_i(\beta) - W_i(\bar \beta)\| ]\leq b_n  \right\},
\end{equation}
where $b_n$ is a sequence such that $b_n\rightarrow 0.$ The equation (\ref{bn_an}) indicates that the sequences $\mathcal{B}_n$ and $b_n$ could be taken such that the radius of
$\mathcal{B}_n$ converges to zero slower than $n^{-1/2}$ and faster than $b_n,$ and
$b_n n^{1/2}/\ln n \rightarrow \infty.$ Then
$\mathbb{P}(\beta_n\in\mathcal{B}_n)\rightarrow 1$ and $\mathbb{P}(\mathcal{E}_n^c)$ decreases to zero faster than any negative power of the sample size $n.$ Hence, in the following  it will suffices to prove  the statements on the set $\{ \beta_n\in\mathcal{B}_n \}\cap \mathcal{E}_n.$

We will focus on $I_n(\beta_n)$ since the arguments for $\hat \omega(\beta_n)$ are similar and much simpler.
Hereafter, by abuse,  we write  $Y_i(t)$ instead of $Y_{ni}(t)$ even when $r_n \neq 0.$
To prove  that
$
I_{n}({\beta_n}) - I_{n} (\bar \beta)= o_{\mathbb{P}}(I_{n}(\bar\beta))$ we will show below that
$
I_{n}({\beta_n}) - I_{n} (\bar \beta)= o_{\mathbb{P}}(n^{-1}h^{-1/2}+r_n^2).$
This shows that $
I_{n}({\beta_n})$ is negligible compared to $I_{n} (\bar \beta)$ both on the null and alternative hypotheses. Indeed, under the null hypothesis, $r_n\equiv 0,$ $\bar \beta=\beta_0$ and $n h^{1/2} I_{n}({\beta_0})$ is asymptotically centered normal distributed, while on the alternative the $I_{n} (\bar \beta)$ is driven by a term of order $r_n^2.$

In the following $C,$ $C'$,... denote  constants that may have different values from line to line.
Let us simplify notation and write
$$
\widehat V_i(\beta)=\widehat {U_{i}w(Z_{i})}(\beta)
$$
and
\begin{equation}\label{simpli}
L_{ij}(\beta)=L_{ij}(\beta,g),\; K_{ij}(\beta)=K_{ij}(\beta,h),\;
\phi _{ij}(\beta) = \phi( W_{i} ({\beta})-W_{j} ({\beta} )).
\end{equation}
Then,
\begin{eqnarray*}
I_{n} ({\beta}) -I_{n} ({\bar{\beta}})\!\! & = & \!\!\dfrac{1}{n(n-1)h}\sum_{i\neq j}\!\!\left[\left\langle \widehat V_i(\beta),\; \widehat V_j(\beta)\right\rangle_{\mathcal{H}} \!- \left\langle \widehat V_i(\bar{\beta}),\; \widehat V_j(\bar{\beta})\right\rangle_{\mathcal{H}}\right]
K_{ij}({\bar{\beta}})\phi _{ij}(\bar{\beta})\\
\!\! & + & \!\!\dfrac{1}{n(n-1)h}\sum_{i\neq j}\!\!\left\langle \widehat V_i(\bar{\beta}),\; \widehat V_j(\bar{\beta})\right\rangle_{\mathcal{H}}\left[K_{ij}({\beta})\phi _{ij}(\beta) - K_{ij}({\bar{\beta}})\phi _{ij}(\bar{\beta}) \right]
\\
&+&\dfrac{1}{n(n-1)h}\sum_{i\neq j}\!\!\left[\left\langle \widehat V_i(\beta),\; \widehat V_j(\beta)\right\rangle_{\mathcal{H}} \!- \left\langle \widehat V_i(\bar{\beta}),\; \widehat V_j(\bar{\beta})\right\rangle_{\mathcal{H}}\right]\\
&&\qquad\qquad\qquad\qquad\qquad\qquad \times
\left[K_{ij}({\beta})\phi _{ij}(\beta) - K_{ij}({\bar{\beta}})\phi _{ij}(\bar{\beta}) \right]\\
&=& D_{n1}(\beta)+D_{n2}(\beta)+D_{n3}(\beta).
\end{eqnarray*}
Let us investigate the uniform rates of $D_{n1}$ and $D_{n2}$, the term  $D_{n3}$ being uniformly smaller. We can write
\begin{eqnarray*}
D_{n1} (\beta)& = &
\dfrac{2}{n(n-1)h}\sum_{i\neq j} \left\langle \widehat V_i(\beta) - \widehat V_i(\bar{\beta}),\; \widehat V_j(\bar{\beta})\right\rangle_{\mathcal{H}}
K_{ij}({\bar{\beta}})\phi _{ij}(\bar{\beta})\\
&&+ \dfrac{1}{n(n-1)h}\sum_{i\neq j} \left\langle \widehat V_i(\beta) - \widehat V_i(\bar{\beta}),\; \widehat V_j(\beta)- \widehat V_j(\bar{\beta})\right\rangle_{\mathcal{H}}
K_{ij}({\bar{\beta}})\phi _{ij}(\bar{\beta})\\
&=& 2D_{n11}  (\beta)+ D_{n12} (\beta).
\end{eqnarray*}
Moreover,
\begin{eqnarray*}
\widehat V_i(\bar{\beta})(t) &=& \frac{1}{n-1}\sum\limits _{k\neq i} [Y_i(t) - Y_k (t)]
\frac{1}{g}L_{ik}(\bar{\beta})
 \\
&=& [Y_i(t)  - r_i(t;\bar{\beta}) ]f_{\bar{\beta}}(X_i^\prime \bar{\beta}) \\
&&+  [Y_i(t) - r_i(t;\bar{\beta})] \left[\frac{1}{n-1}\sum\limits _{k\neq i} \frac{1}{g}L_{ik}(\bar{\beta})- f_{\bar{\beta}}(X_i^\prime \bar{\beta})\right]\\
&&+ \frac{1}{n-1}\sum\limits _{k\neq i} \{r_i(t;\bar{\beta}) - r_k(t;\bar{\beta})\}\frac{1}{g}L_{ik}(\bar{\beta})\\
&&- \frac{1}{n-1}\sum\limits _{k\neq i} \{ Y_k (t) - r_k(t;\bar{\beta})\}\frac{1}{g}L_{ik}(\bar{\beta})\\
&=& [Y_i(t)  - r_i(t;\bar{\beta}) ]f_{\bar{\beta}}(X_i^\prime \bar{\beta}) + [Y_i(t) - r_i(t;\bar{\beta})] R_{1,ni} + R_{2,ni}(t) - R_{3,ni}(t)  ,
\end{eqnarray*}
where, by  Lemma \ref{deltas2}
$$
\sup_{1\leq i \leq n}\!\!|R_{1,ni}| = O_{\mathbb{P}}(g + n^{-1/2} g^{-1/2}\ln^{1/2} n),
$$
and, Lemma \ref{Deltas} yields
$$
\sup_{1\leq i \leq n}\sup_{t\in[0,1]}|R_{3,ni}(t)|  = O_{\mathbb{P}}(n^{-1/2} g^{-1/2}\ln^{1/2} n ).
$$
A representation of $R_{2,ni}(t)$ is provided in  Lemma \ref{deltas3}.
On the other hand,
\begin{eqnarray*}
\widehat V_i(\beta)(t) - \widehat V_i(\bar{\beta})(t) &=& \frac{1}{n-1}\sum\limits _{k\neq i} [Y_i(t) - Y_k (t)]\left[
\frac{1}{g}L_{ik}(\beta) -
\frac{1}{g}L_{ik}(\bar{\beta}) \right].
\end{eqnarray*}

\subsubsection*{Uniform bounds for  $D_{n1}.$}
\textbf{The rate of $D_{n11}.$} Since $
Y_i(t) = r_i(t;\bar{\beta}) + r_n \delta(X_i,t) + \epsilon_i(t),
$
with $\mathbb{E}[\epsilon_i (t) \mid X_i] =0,$ we have $D_{n11}(\beta) =  D_{n111}(\beta) + R_{n11}(\beta)$ with
\begin{eqnarray*}
D_{n111}(\beta)&=&\dfrac{1}{n(n-1)^2h}\sum_{i\neq j\neq k} \left\langle Y_i(\cdot) - Y_k (\cdot),\; Y_j(\cdot)  - r_j(\cdot;\bar{\beta}) \right\rangle_{\mathcal{H}}f_{\bar{\beta}}(X_j^\prime \bar{\beta})\\
&& \qquad\qquad \qquad\qquad \qquad \times\left[
\frac{1}{g}L_{ik}(\beta) - \frac{1}{g}L_{ik}(\bar \beta) \right] K_{ij}({\bar{\beta}})\phi _{ij}(\bar{\beta})\\
&=& \dfrac{1}{n(n-1)^2h}\sum_{i\neq j\neq k} \left\langle Y_i(\cdot) - Y_k (\cdot),\; \epsilon_j(\cdot) \right\rangle_{\mathcal{H}}f_{\bar{\beta}}(X_j^\prime \bar{\beta})
\\
&& \qquad\qquad \qquad\qquad \qquad \times\left[
\frac{1}{g}L_{ik}(\beta) - \frac{1}{g}L_{ik}(\bar \beta) \right]  K_{ij}({\bar{\beta}})\phi _{ij}(\bar{\beta}),
\end{eqnarray*}
and $R_{n11}(\beta)= D_{n11}(\beta) -  D_{n111}(\beta) .$ We decompose
\begin{eqnarray*}
D_{n111} (\beta)&=& \dfrac{1}{n(n-1)^2h}\sum_{i\neq j\neq k} \left\langle \epsilon_i(\cdot) - \epsilon_k (\cdot),\; \epsilon_j(\cdot) \right\rangle_{\mathcal{H}}f_{\bar{\beta}}(X_j^\prime \bar{\beta})\\
&& \qquad\qquad \qquad\times \left[
\frac{1}{g}L_{ik}(\beta) - \frac{1}{g}L_{ik}(\bar{\beta}) \right] K_{ij}({\bar{\beta}})\phi _{ij}(\bar{\beta}) \\
& & + \dfrac{1}{n(n-1)^2h}\sum_{i\neq j\neq k} \left\langle r_i(\cdot;\bar\beta) - r_k(\cdot;\bar\beta),\; \epsilon_j(\cdot)\right\rangle_{\mathcal{H}}f_{\bar{\beta}}(X_j^\prime \bar{\beta})\\
&& \qquad\qquad \qquad\times \left[
\frac{1}{g}L_{ik}(\beta) - \frac{1}{g}L_{ik}(\bar{\beta}) \right] K_{ij}({\bar{\beta}})\phi _{ij}(\bar{\beta})
\\
& & + \dfrac{r_n}{n(n-1)^2h}\sum_{i\neq j\neq k} \left\langle \delta(X_i,t) - \delta(X_k,t),\; \epsilon_j(\cdot)\right\rangle_{\mathcal{H}}f_{\bar{\beta}}(X_j^\prime \bar{\beta})\\
&& \qquad\qquad \qquad\times \left[
\frac{1}{g}L_{ik}(\beta) - \frac{1}{g}L_{ik}(\bar{\beta}) \right] K_{ij}({\bar{\beta}})\phi _{ij}(\bar{\beta})
\\
&=& D_{n1111}(\beta)+D_{n1112}(\beta)+ r_n D_{n1113}(\beta).
\end{eqnarray*}
The quantity $gh D_{n1111}(\beta)$ could be decomposed in a sum of degenerate $U-$process of order 3 and another one of order 2 indexed by $\beta.$ To bound them  we use the maximal inequality of \cite{Sherman1994}. Since $nh^2, ng^{4} \rightarrow \infty,$ deduce that the degenerate $U-$process of order 3 is of uniform rate
$$
n^{-3/2} O_{\mathbb{P}} ( h^{\alpha/2} \{ b_n^2 g^{-1} \}^{\alpha/2}   ) = gh \times o_{\mathbb{P}} ( n^{-1} h^{-1/2}),
$$
over any sequence of balls centered at $\bar \beta$ with radius decreasing to zero faster than $b_n,$  where $b_n$ is a sequence such that $b_n n^{1/2}/\ln n \rightarrow \infty$ and $\alpha$ could be a number in the interval $(0,1)$ arbitrarily close to 1. The details on how the maximal inequality of \cite{Sherman1994} applies are provided  below for deriving the uniform rate of $D_{n12}.$
To bound the right-hand side term in that maximal inequality we use the fact that $\mathbb{E}(\| \epsilon \| _{\mathcal{H}}^2\mid X)$ and $f_{\bar{\beta}}(X^\prime \bar{\beta})$ are bounded and the uniform bounds (\ref{eeqq3}), (\ref{eeqq4}) and (\ref{eeqq5}) from Lemma \ref{mon_tric} in the Appendix.
Using very similar arguments, the degenerate $U-$process of order 2 in the decomposition of $gh D_{n1111}(\beta)$  could be shown to be of uniform rate
$$
  n^{-1} O_{\mathbb{P}} ( h^{\alpha/2} \{ b_n^2 g^{-1}  \}^{\alpha/2}    )=
 gh \times o_{\mathbb{P}} ( n^{-1} h^{-1/2})
$$
provided that $nh^2, ng^{4}\rightarrow \infty$  and $\alpha$ is sufficiently close to 1.
Next, for $ng D_{n1112}(\beta),$ that is centered, use the Hoeffding decomposition and the regularity of the function $v\mapsto \mathbb{E}[ Y(t)\mid Z(\bar \beta) = v ]. $
For  the   degenerate $U-$processes of order 3 and 2 in the Hoeffding decomposition of  $D_{n1112}(\beta)$ we apply  the maximal inequality of \cite{Sherman1994} as previously. Deduce the respective uniform rates over $\mathcal{B}_n$
$$
 g^2  n^{-3/2} O_{\mathbb{P}} ( h^{\alpha/2} \{ b_n^2 g^{-1}  \}^{\alpha/2}   ) = gh\times o_{\mathbb{P}} ( n^{-1} h^{-1/2}),
$$
and
$$
  g^2  n^{-1} O_{\mathbb{P}} (  h^{\alpha/2} \{ b_n^2 g^{-1}  \}^{\alpha/2}      ) = gh\times o_{\mathbb{P}} ( n^{-1} h^{-1/2}).
$$
It remains the $U-$process of order 1. Using again the bounds from Lemma \ref{mon_tric}, deduce the uniform rate over $\mathcal{B}_n$
$$
  g^2  n^{-1/2} O_{\mathbb{P}} (  h^{\alpha} \{ b_n^2 g^{-1}  \}^{\alpha/2}      ) = gh\times o_{\mathbb{P}} ( n^{-1} h^{-1/2}).
$$
Deduce $D_{n1112}(\beta_n)=o_{\mathbb{P}} (n^{-1}h^{-1/2}).$ For $gh D_{n1113}(\beta)$ the arguments are similar, but without the $g^2$ factor, and yield the uniform rate
$$
n^{-1/2} O_{\mathbb{P}} (  h^{\alpha} \{ b_n^2 g^{-1}  \}^{\alpha/2}      ) =  gh\times o_{\mathbb{P}} ( n^{-1} h^{-1/2}) = o_{\mathbb{P}} ( n^{-1/2}h^{-1/4}),
$$
provided $nh^2, ng^{4}\rightarrow\infty$ and $\alpha$ is sufficiently close to 1.
Deduce that $D_{n111}(\beta_n) =  o_{\mathbb{P}} ( n^{-1} h^{-1/2} +r_n^2).$

For $R_{n11}(\beta)$ we can write
\begin{eqnarray*}
R_{n11}(\beta) &=& \dfrac{1}{n(n-1)^2h}\sum_{i\neq j} \left\langle Y_i(\cdot) - Y_j (\cdot),\; Y_j(\cdot)  - r_j(t;\bar{\beta}) \right\rangle_{\mathcal{H}}f_{\bar{\beta}}(X_j^\prime \bar{\beta})\\
&& \qquad\qquad \qquad\times \left[
\frac{1}{g}L_{ij}(\beta) - \frac{1}{g}L_{ij}(\bar{\beta}) \right] K_{ij}({\bar{\beta}})\phi _{ij}(\bar{\beta}) \\
& & + \dfrac{1}{n(n-1)^2h}\sum_{i\neq j, i\neq k} \left\langle Y_i(\cdot) - Y_k (\cdot),\; Y_j(\cdot)  - r_j(t;\bar{\beta}) \right\rangle_{\mathcal{H}}R_{1,nj}\\
&& \qquad\qquad \qquad\times \left[
\frac{1}{g}L_{ik}(\beta) - \frac{1}{g}L_{ik}(\bar{\beta}) \right] K_{ij}({\bar{\beta}})\phi _{ij}(\bar{\beta})
\\
& & + \dfrac{1}{n(n-1)^2h}\sum_{i\neq j, i\neq k} \left\langle Y_i(\cdot) - Y_k (\cdot),\; R_{2,nj}(\cdot)+
R_{3,nj}(\cdot)\right\rangle_{\mathcal{H}}\\
&& \qquad\qquad \qquad\times \left[
\frac{1}{g}L_{ik}(\beta) - \frac{1}{g}L_{ik}(\bar{\beta}) \right] K_{ij}({\bar{\beta}})\phi _{ij}(\bar{\beta})
\\
&=& R_{n111}(\beta)+R_{n112}(\beta)+R_{n113}(\beta).
\end{eqnarray*}
We only investigate $R_{n111} (\beta)$, the terms $R_{n112} (\beta)$ and $R_{n113} (\beta)$ are uniformly smaller compared to $D_{n111} (\beta)$.
We can write
\begin{eqnarray*}
R_{n111} (\beta)&=& \dfrac{1}{n-1} \; \dfrac{1}{n(n-1)h}\sum_{i\neq j} \left\langle \epsilon_i(\cdot) - \epsilon_j (\cdot),\; \epsilon_j(\cdot) \right\rangle_{\mathcal{H}}f_{\bar{\beta}}(X_j^\prime \bar{\beta})\\
&& \qquad\qquad \qquad\times \left[
\frac{1}{g}L_{ij}(\beta) - \frac{1}{g}L_{ij}(\bar{\beta}) \right] K_{ij}({\bar{\beta}})\phi _{ij}(\bar{\beta}) \\
& & + \dfrac{1}{n-1} \;\dfrac{1}{n(n-1)h}\sum_{i\neq j} \left\langle r_i(\cdot;\bar\beta) - r_j(\cdot;\bar\beta),\; \epsilon_j(\cdot)\right\rangle_{\mathcal{H}}f_{\bar{\beta}}(X_j^\prime \bar{\beta})\\
&& \qquad\qquad \qquad\times \left[
\frac{1}{g}L_{ij}(\beta) - \frac{1}{g}L_{ij}(\bar{\beta}) \right] K_{ij}({\bar{\beta}})\phi _{ij}(\bar{\beta})
\\
& & + \dfrac{1}{n-1} \; \dfrac{r_n}{n(n-1)h}\sum_{i\neq j} \left\langle \delta(X_i,t) - \delta(X_j,t),\; \epsilon_j(\cdot)\right\rangle_{\mathcal{H}}f_{\bar{\beta}}(X_j^\prime \bar{\beta})\\
&& \qquad\qquad \qquad\times \left[
\frac{1}{g}L_{ij}(\beta) - \frac{1}{g}L_{ij}(\bar{\beta}) \right] K_{ij}({\bar{\beta}})\phi _{ij}(\bar{\beta})
\\
&=& R_{n1111}(\beta)+R_{n1112}(\beta)+ r_n R_{n1113}(\beta).
\end{eqnarray*}
The leading term in $R_{n1111}(\beta)$
is
$$
\dfrac{1}{n-1} \dfrac{1}{n(n-1)h}\sum_{i\neq j} \left\|\epsilon_j (\cdot) \right\|^2_{\mathcal{H}}f_{\bar{\beta}}(X_j^\prime \bar{\beta}) \left[
\frac{1}{g}L_{ij}(\beta) - \frac{1}{g}L_{ij}(\bar{\beta}) \right] K_{ij}({\bar{\beta}})\phi _{ij}(\bar{\beta}).
$$
Use the boundedness of $\mathbb{E}[\left\|\epsilon_j (\cdot) \right\|^2_{\mathcal{H}}\mid X_j]$ and
$f_{\bar{\beta}}(X_j^\prime \bar{\beta}),$
and Lemma \ref{mon_tric} to deduce that $R_{n1111}(\beta_n)=o_{\mathbb{P}} (n^{-1}).$
Gathering facts deduce that
$$
D_{n11} (\beta_n) =  o_{\mathbb{P}}(n^{-1}h ^{-1/2} + r_n^2).
$$

\quad

\textbf{The rate of $D_{n12}.$} We have
$$
\widehat V_i(\beta)(t) - \widehat V_i(\bar{\beta})(t) = Y_i(t) \Delta_{1,ni}(\beta) +  \Delta_{2,ni}(\beta)
$$
with $ \Delta_{1,ni}(\beta) $ and $   \Delta_{2,ni}(\beta)$ independent of $t$ and
$$\sup_{1\leq i\leq n } \sup_{\beta\in\mathcal{B}_n } \{  |\Delta_{1,ni}| +  |\Delta_{2,ni}|  \} = O_{\mathbb{P}}(n^{-1/2}g^{-1/2} \ln^{1/2}n + b_n );$$
see Lemma \ref{Deltas}. Replacing and taking absolute values, deduce
$$
D_{n12}(\beta_n) = O_{\mathbb{P}}(n^{-1}g^{-1} \ln n + n^{-1} \ln^2 n) = o_{\mathbb{P}}(n^{-1}h^{-1/2}),
$$
since $g^{-1} h^{1/2} \rightarrow 0$ and $h\ln^4 n \rightarrow 0.$

Gathering facts deduce that
$$
D_{n1} (\beta_n) = D_{n11} (\beta_n) + D_{n12} (\beta_n) = o_{\mathbb{P}}(n^{-1}h ^{-1/2} + r_n^2).
$$


\subsubsection*{Uniform bounds for  $D_{n2}.$}
We have
\begin{eqnarray*}
D_{n2}(\beta) &=& \dfrac{1}{n(n-1)h}\sum_{i\neq j}\!\!\left\langle \widehat V_i(\bar{\beta}),\; \widehat V_j(\bar{\beta})\right\rangle_{\mathcal{H}}\left[K_{ij}({\beta})\phi _{ij}(\beta) - K_{ij}({\bar{\beta}})\phi _{ij}(\bar{\beta}) \right]\\
&=& \dfrac{1}{n(n-1)h}\sum_{i\neq j}\!\!\left\langle  [Y_i(t)  - r_i(t;\bar{\beta}) ]f_{\bar{\beta}}(X_i^\prime \bar{\beta}),\;  [Y_j(t)  - r_j(t;\bar{\beta}) ]f_{\bar{\beta}}(X_j^\prime \bar{\beta})\right\rangle_{\mathcal{H}}\\
&&\qquad\qquad\qquad\qquad\qquad\qquad \times \left[K_{ij}({\beta})\phi _{ij}(\beta) - K_{ij}({\bar{\beta}})\phi _{ij}(\bar{\beta}) \right]\\
&&+ \text{terms of smaller rate}\\
&=& D_{n21} (\beta)+ \text{terms of smaller rate}.
\end{eqnarray*}
Recall that by construction,
$$
\mathbb{E}[Y_i(t) \mid X_i] = r_i(t;\bar{\beta}) + r_n \delta(X_i,t),
$$
so that
$$
[Y_i(t)  - r_i(t;\bar{\beta}) ]f_{\bar{\beta}}(X_i^\prime \bar{\beta}) = [\epsilon_i(t)+ r_n \delta(X_i,t)]f_{\bar{\beta}}(X_i^\prime \bar{\beta}),
$$
with $\mathbb{E}[\epsilon_i(t)\mid X_i ]=0,$ $\forall t\in [0,1].$ Thus
\begin{multline*}
D_{n21}(\beta) = \dfrac{1}{n(n-1)h}\sum_{i\neq j}\!\!\left\langle  \epsilon_i(t),\; \epsilon_j(t)\right\rangle_{\mathcal{H}}f_{\bar{\beta}}(X_i^\prime \bar{\beta}) f_{\bar{\beta}}(X_j^\prime \bar{\beta})\left[K_{ij}({\beta})\phi _{ij}(\beta) - K_{ij}({\bar{\beta}})\phi _{ij}(\bar{\beta}) \right]\\
+ \dfrac{2r_n}{n(n-1)h}
\sum_{i\neq j}\!\left\langle  \epsilon_i(t),\; \delta(X_j,t) \right\rangle_{\mathcal{H}}f_{\bar{\beta}}(X_i^\prime \bar{\beta}) f_{\bar{\beta}}(X_j^\prime \bar{\beta})\left[K_{ij}({\beta})\phi _{ij}(\beta) - K_{ij}({\bar{\beta}})\phi _{ij}(\bar{\beta}) \right]\\
+ \dfrac{r^2_n}{n(n-1)h}
\sum_{i\neq j}\!\left\langle  \delta(X_i,t),\; \delta(X_j,t) \right\rangle_{\mathcal{H}}f_{\bar{\beta}}(X_i^\prime \bar{\beta}) f_{\bar{\beta}}(X_j^\prime \bar{\beta})\left[K_{ij}({\beta})\phi _{ij}(\beta) - K_{ij}({\bar{\beta}})\phi _{ij}(\bar{\beta}) \right]
\\
= D_{n211} (\beta)+ 2r_n D_{n212} (\beta) + r_n^2 D_{n213} (\beta).
\end{multline*}
The term $D_{n211} (\cdot)$ is a degenerate $U-$process of order 2, indexed by $\beta$.
Consider  the family of functions
\begin{equation}\label{fn_n}
\mathcal{F}_n =\{h(\cdot,\cdot;\beta):  \beta\in\mathcal{B}_n \}
\end{equation}
with
$$
h((x_1,\epsilon_1), (x_2,\epsilon_2);\beta) = \langle \epsilon_1(\cdot),\;  \epsilon_2 (\cdot)\rangle_{\mathcal{H}}  f_{\bar{\beta}}(x_1^\prime \bar{\beta}) f_{\bar{\beta}}(x_2^\prime \bar{\beta})[K_{12}(\beta)\phi_{12}(\beta)
-K_{12}(\bar{\beta})\phi_{12}(\bar{\beta} )].
$$
It is quite easy to see that $\mathcal{F}_n$ is a VC class, or Euclidean in the terminology of \cite{Sherman1994}, for a squared integrable envelope $H(\cdot),$ with some $A$ and $V$ independent of $n$. (Recall that the $\delta-$covering number of an Euclidean class of function is bounded by $A\delta^{-V}$.)
Since $\mathbb{E} (\|\epsilon_1 (\cdot)\|^2_{\mathcal{H}}  \mid X_1)$ and $f_{\bar{\beta}}(X_1^\prime \bar{\beta})$ are bounded, and the kernel $K$ is bounded, by Lemma \ref{mon_tric} deduce that
$$
\mathbb{E} \left[\sup_{\beta\in\mathcal{B}_n} h (\cdot,\cdot; \beta)^2\right]\leq C h^{1/2}b_n
$$
for some constant $C>0$ independent on $n$ and $\bar{\beta}$. See Lemma \ref{mon_tric} below.
Applying the Main Corollary of \cite{Sherman1994} with $k=2,$ $p=1$, deduce that\footnote[2]{Let us point out that the rate could be improved if one tracks the dependence of the constants appearing in Sherman's result on the $\delta-$covering number of $\mathcal{F}_n.$ This covering number decreases with $n$ as the parameter set $\mathcal{B}_n$ shrinks to $\bar\beta$. For our purposes we do not need this refinement.}
$$
\sup_{\beta} |h D_{n211}(\beta)| \leq \frac{C'}{n} (b_n h^{1/2} )^{\alpha/2} = n^{-1}h^{1/2} \times O(n^{-\alpha/4} h^{\alpha/4-1/2})
$$
for  $0<\alpha<1$.
Since and $\alpha$ could be arbitrarily
close to 1 and $b_n$ could be any sequence such that $b_nn^{1/2} \ln^{-1} n\rightarrow \infty$ and $nh^{3/2}\rightarrow 0,$ deduce that
$$
D_{n211}(\beta_n) =o_{\mathbb{P}}(n^{-1}h^{-1/2}).
$$
For the uniform rate of the centered $U-$process $D_{n212}(\cdot),$ use the Hoeffding decomposition. The degenerate $U-$process of order 1 in this decomposition could be handled with the arguments used for $D_{n211}(\cdot)$ and shown to be of uniform rate $o_{\mathbb{P}}(n^{-1/2})$. The degenerate $U-$process of order 2 in the decomposition is
$$
D_{n212,1}(\beta) = \dfrac{1}{n}
\sum_{i\neq j}\!\left\langle  \epsilon_i(t),\; \gamma_i(\beta,t;h) \right\rangle_{\mathcal{H}}f_{\bar{\beta}}(X_i^\prime \bar{\beta})
$$
where
$$
\gamma_i(\beta,t;h) = \mathbb{E}\left\{\delta(X_j,t) f_{\bar{\beta}}(X_j^\prime \bar{\beta})h^{-1} \left[K_{ij}({\beta})\phi _{ij}(\beta) - K_{ij}({\bar{\beta}})\phi _{ij}(\bar{\beta}) \right]\mid X_i \right\}.
$$
Since  $f_{\bar{\beta}}$ and $\delta(X,\cdot)$ are supposed bounded, arguments as used for Lemma \ref{mon_tric} allow to show that $
\mathbb{E}[\sup_\beta \gamma^2_i(\beta,t)] = o(1).$ Deduce that $D_{n212,1}(\beta_n) = o_{\mathbb{P}}(n^{-1/2}).$ Gathering facts, $r_n D_{n212} (\beta_n)$ is negligible compared to $r_n^2.$ By the same arguments, $D_{n213} (\beta_n) = o_{\mathbb{P}}(1)$ so that we can conclude that
$
D_{n2} (\beta_n) = o_{\mathbb{P}}(n^{-1}h ^{-1/2} + r_n^2).
$
\end{proof}

\quad

\begin{proof}[Proof of Proposition \ref{sam_1mai}]
Let us consider the simplified notation from equation (\ref{simpli}) and further simplify in the case $\beta=\beta_0$ and write
\begin{equation}\label{simpli2}
L_{ij} = L_{ij}(\beta_0,g), \quad K_{ij} = K_{ij}(\beta_0,h),\quad \text{and} \quad
\phi_{ij} = \phi(W_i(\beta_0)- W_j(\beta_0)).
\end{equation}
Notice that
\begin{eqnarray*}
I_{n}^{\{l\}}\left(\beta_{0}\right) & = & \dfrac{1}{n\left(n-1\right)h}\sum_{1\leq i\neq j\leq n}\left\{ \left\langle \left(r_{i}-\tilde{r}_{i}\right)\left(\cdot;\,\beta_{0}\right),\,\left(r_{j}-\tilde{r}_{j}\right)\left(\cdot;\,\beta_{0}\right)\right\rangle _{L^{2}}\right.\\
 &  & \hphantom{\dfrac{1}{n\left(n-1\right)h}\sum_{1\leq i\neq j\leq n}\quad}+\left.\left\langle \epsilon_{i}\left(\cdot\right),\,\epsilon_{j}\left(\cdot\right)\right\rangle _{L^{2}}\right.\\
 &  & \hphantom{\dfrac{1}{n\left(n-1\right)h}\sum_{1\leq i\neq j\leq n}\quad}+\left.\left\langle \tilde{\epsilon}_{i}\left(\cdot\right),\,\tilde{\epsilon}_{j}\left(\cdot\right)\right\rangle _{L^{2}}\right.\\
 &  & \hphantom{\dfrac{1}{n\left(n-1\right)h}\sum_{1\leq i\neq j\leq n}\quad}+\left.2\left\langle \epsilon_{i}\left(\cdot\right),\,\left(r_{j}-\tilde{r}_{j}\right)\left(\cdot;\,\beta_{0}\right)\right\rangle _{L^{2}}\right.\\
 &  & \hphantom{\dfrac{1}{n\left(n-1\right)h}\sum_{1\leq i\neq j\leq n}\quad}-\left.2\left\langle \tilde{\epsilon}_{i}\left(\cdot\right),\,\left(r_{j}-\tilde{r}_{j}\right)\left(\cdot;\,\beta_{0}\right)\right\rangle _{L^{2}}\right.\\
 &  & \hphantom{\dfrac{1}{n\left(n-1\right)h}\sum_{1\leq i\neq j\leq n}\quad}-\left.2\left\langle \epsilon_{i}\left(\cdot\right),\,\tilde{\epsilon}_{j}\left(\cdot\right)\right\rangle _{L^{2}}\right\} \hat{f}_{\beta_{0},i}\hat{f}_{\beta_{0},j}K_{ij}\phi_{ij}\\
 & = & I_{1}\left(\beta_{0}\right)+I_{2}\left(\beta_{0}\right)+I_{3}\left(\beta_{0}\right)+2I_{4}\left(\beta_{0}\right)-2I_{5}\left(\beta_{0}\right)-2I_{6}\left(\beta_{0}\right)
\end{eqnarray*}
with
$$
\hat{f}_{\beta,i}=\dfrac{1}{\left(n-1\right)g}\sum_{k\neq i}L_{ik}\left(\beta\right)
,\quad r_{i}\left(t;\,\beta\right)=\mathbb{P}\left[Y_{i}\leq\Phi^{-1}\left(t\right)\mid X_{i}^{\prime}\beta\right],
$$
$$\tilde{r}_{i}\left(t;\,\beta\right)=\dfrac{1}{\left(n-1\right)g\hat{f}_{\beta,i}}\sum_{k\neq i}r_{k}\left(t;\,\beta\right)L_{ik}\left(\beta\right)$$
and $\tilde{\epsilon}_{i}\left(\cdot\right)$ is defined as $\tilde{r}_{i}\left(t;\,\beta\right)$
by replacing $r_{i}\left(t;\,\beta\right)$ by $\epsilon_{i}\left(\cdot\right)$.
This decomposition of $I_{n}^{\{l\}}\left(\beta_{0}\right) $ is given by the identity
\[
\widehat{U_{i}\omega\left(Z_{i}\right)}\left(\cdot;\,\beta_{0}\right)=\left[r_{i}\left(\cdot;\,\beta_{0}\right)-\tilde{r}_{i}\left(\cdot;\,\beta_{0}\right)+\epsilon_{i}\left(\cdot\right)-\tilde{\epsilon}_{i}\left(\cdot\right)\right]\hat{f}_{\beta_{0},i}.
\]
The terms $I_{1}\left(\beta_{0}\right)$ and $I_{3}\left(\beta_{0}\right)$
are treated in Lemmas \ref{prop:I1} and \ref{prop:I3} in Section \ref{tech_lemm}.
For $I_{2}\left(\beta_{0}\right)$, let
us introduce
\[
\omega_{n}^{2}\left(\beta\right)=\dfrac{2}{n\left(n-1\right)h}\sum_{i=1}^{n}\sum_{j\neq i}\intop\intop\Gamma^2\left(s,t\right)ds\, dt\,\hat{f}_{\beta,i}^{2}\hat{f}_{\beta,j}^{2}K_{ij}^{2}\left(\beta\right)\phi_{ij}^{2}\left(\beta\right).
\]
Proposition \ref{prop:I2} below ensures that $nh^{1/2}\omega_{n}^{-1}\left(\beta_{0}\right)I_{2}\left(\beta_{0}\right)\to\mathcal{N}\left(0,1\right)$ in law.
The terms $I_{4}\left(\beta_{0}\right)$, $I_{5}\left(\beta_{0}\right)$
and $I_{6}\left(\beta_{0}\right)$
can be shown to be negligible in a similar way as $I_{1}\left(\beta_{0}\right)$ and
$I_{3}\left(\beta_{0}\right)$. Lemma \ref{lem:VarLimit} shows that
$\omega_{n}^{2}\left(\beta_{0}\right)\rightarrow\omega^{2}\left(\beta_{0}\right)$ in probability with $\omega^{2}\left(\beta_{0}\right)>0$
and thus $I_{j}\left(\beta_{0}\right)/\omega_{n}\left(\beta_{0}\right)$
is of the same order as $I_{j}\left(\beta_{0}\right)$ for $j\in\left\{ 1,3,4,5,6\right\} $. Finally, it is easy to check that $\omega_{n}\left(\beta_0\right) - \hat \omega_{n}^{\{l\}}\left(\beta_0\right)= o_{\mathbb{P}}(1).$
Then the result of the proposition follows.
\end{proof}

\quad

\begin{proposition}
\label{prop:I2}Under the conditions of Proposition \ref{sam_1mai}, $$
nh^{1/2}\omega_{n}^{-1}\left(\beta_{0}\right)
I_{2}\left(\beta_{0}\right)\to\mathcal{N}\left(0,1\right)\quad \text{in law}.$$
\end{proposition}

\medskip

\begin{proof}
$\left\{ S_{n,m},\,\mathcal{F}_{n,m},\,1\leq m\leq n,\, n\geq1\right\} $
is a martingale array with $S_{n,1}=0$ and
\[
S_{n,m}\left(\beta_{0}\right)=\sum_{i=1}^{m}G_{n,i}\left(\beta_{0}\right)
\]
with
\[
G_{n,i}\left(\beta_{0}\right)=\dfrac{2h^{p/2}}{\omega_{n}\left(n-1\right)h}\left\langle \epsilon_{i}\left(\cdot\right)\hat{f}_{\beta_{0},i},\,\sum_{j=1}^{i-1}\epsilon_{j}\left(\cdot\right)\hat{f}_{\beta_{0},j}K_{ij}\phi_{ij}\right\rangle _{L^{2}}
\]
and $\mathcal{F}_{n,m}$ is the $\sigma$-field generated by $\left\{ X_{1},\dots,\, X_{n},\, Y_{1},\dots,\, Y_{m}\right\} $.
Thus $$nh^{1/2}\omega_{n}^{-1}\left(\beta_{0}\right)I_{2}\left(\beta_{0}\right)=S_{n,n}\left(\beta_{0}\right).$$
Next, define
\begin{eqnarray*}
V_{n}^{2}\left(\beta_{0}\right) & = & \sum_{i=2}^{n}\mathbb{E}\left[G_{n,i}^{2}\left(\beta_{0}\right)\mid\mathcal{F}_{n,i-1}\right]\\
 & = & \dfrac{4}{\omega_{n}^{2}\left(n-1\right)^{2}h}\sum_{i=2}^{n}\int\int\Gamma\left(s,t\right)\hat{f}_{\beta_{0},i}^{2}\left(\sum_{j=1}^{i-1}\epsilon_{j}\left(s\right)\hat{f}_{\beta_{0},j}K_{ij}\phi_{ij}\right)\\
 &  & \qquad\qquad\qquad\qquad\qquad\qquad\times\left(\sum_{k=1}^{i-1}\epsilon_{k}\left(t\right)\hat{f}_{\beta_{0},k}K_{ik}\phi_{ik}\right)ds\, dt
\end{eqnarray*}
and decompose
\begin{multline}\label{an_bn}
V_{n}^{2}\left(\beta_{0}\right) = \dfrac{4}{\omega_{n}^{2}\left(n-1\right)^{2}h}\sum_{i=2}^{n}\sum_{j=1}^{i-1}\int\int\Gamma\left(s,t\right)\hat{f}_{\beta_{0},i}^{2}\epsilon_{j}\left(s\right)\epsilon_{j}\left(t\right)\hat{f}_{\beta_{0},j}^{2}K_{ij}^{2}\phi_{ij}^{2}ds\, dt\\
 +\dfrac{8}{\omega_{n}^{2}\left(n-1\right)^{2}h}\sum_{i=3}^{n}\sum_{j=2}^{i-1}\sum_{k=1}^{j-1}\int\int\Gamma\left(s,t\right)\hat{f}_{\beta_{0},i}^{2}\epsilon_{j}\left(s\right)\epsilon_{k}\left(t\right)\hat{f}_{\beta_{0},j}\hat{f}_{\beta_{0},k}K_{ij}^{2}\phi_{ij}^{2}ds\, dt\\
  =  A_{n}\left(\beta_{0}\right)+B_{n}\left(\beta_{0}\right).
\end{multline}
From Lemma \ref{lem:CondVar}, we have that the martingale array satisfies
Corollary 3.1 of \cite{Hall1980} and the result
follows.\end{proof}

\quad

\setcounter{equation}{0}

\section{Appendix 2: technical lemmas}\label{tech_lemm}

\qquad

In the following results the kernels $L$ and $K$ are supposed to satisfy the conditions of Assumption \ref{ass_app}-(f).

\begin{lem}\label{Deltas}
Assume that $\mathbb{E}[\exp(a\|X\|)]<\infty$ for some $a>0.$ Consider that $g\rightarrow 0$ and $ng^{4/3} /\ln n \rightarrow \infty.$ For any $t\in[0,1]$ let $Y_k(t),$ $1\leq k \leq n,$ be an i.i.d. random variables like in the proof of Proposition \ref{as_equiv} such that $\mathbb{E}[\sup_t |Y_k(t) |^a ]<\infty$ for some $a>8.$ Moreover, assume that the maps $v\mapsto \mathbb{E}[|Y_k (t) | \mid X^\prime \bar\beta = v ]f_{\bar\beta}(v),$ $v\in\mathbb{R},$ $t\in[0,1],$ are uniformly Lipschitz (the Lipschitz constant does not depend on $t$). Then
$$
 \max_{1\leq i\leq n}\sup_{t\in[0,1] } \sup_{\beta\in\mathcal{B}_n } \left| \frac {1}{n-1} \sum_{k\neq i} Y_k (t)\frac{1}{g}\left[L_{ik} (\beta) - L_{ik} (\bar \beta)\right] \right| = O_{\mathbb{P}}(n^{-1/2}g^{-1/2} \ln^{1/2}n + b_n ).
$$
Moreover,
$$
 \max_{1\leq i\leq n}\sup_{t\in[0,1] } \sup_{\beta\in\mathcal{B}_n } \left| \frac{1}{n-1}\sum\limits _{k\neq i} \{ Y_k (t) - \mathbb{E}[Y_k (t)\mid X_k^\prime \bar\beta] \}\frac{1}{g}L_{ik}(\bar{\beta}) \right| = O_{\mathbb{P}}(n^{-1/2}g^{-1/2} \ln^{1/2}n  ).
 $$
\end{lem}

\begin{proof}[Proof of Lemma \ref{Deltas}] Recall that
$Y_i(t) \equiv Y_i$ (in the case of SIM for mean regression) or $Y_i(t) =\mathbf{1}\{Y_i\leq \Phi^{-1}(t)\} $ (for the case of single-index assumption on the conditional law), and  $r_i(t;\bar{\beta}) = \mathbb{E}[Y_i(t) \mid Z(\bar{\beta})],$ $t\in[0,1].$
For any $t\in[0,1]$ we decompose
\begin{eqnarray*}
\frac {1}{ng} \sum_{k\neq i} Y_k (t) L_{ik} (\beta) \!\! &=& \!\! \frac{1}{ng} \sum_{k=1}^n
\left\{Y_k (t)L\left( (X_i-X_k)^\prime \beta /g \right) -\mathbb{E}\left[Y(t) L\left(
  (X_i-X)^\prime \beta /g \right)\mid X_i\right]\right\}\\ && + \mathbb{E}\left[Y(t)
  g^{-1} L\left( (X_i-X)^\prime \beta /g\right)\mid X_i\right] - n^{-1}g^{-1} L(0) Y_i (t)\\ &=&
\Sigma_{1ni}(\beta,t) + \Sigma_{2i} (\beta,t) - n^{-1}g^{-1} L(0) Y_i (t).
\end{eqnarray*}
The moment condition on $Y$ guarantees that $\max_{1\leq i\leq n} \sup_{t} |Y_i (t)| = o_{\mathbb{P}}(n^b)$ for some $0<b<1/8.$ This and the fact that
 $ng^{4/3} /\ln n \rightarrow \infty$ make that $\max_{1\leq i\leq n} \sup_{t} n^{-1}g^{-1}|Y_i (t)| = o_{\mathbb{P}}(n^{-1/2}g^{-1/2} \ln^{1/2}n ).$ On the other hand, by Lemma \ref{deltas4},
 $$
\max_{1\leq i\leq n}\sup_{t\in [0,1]} \sup_{\beta\in\mathcal{B}_n }\left|\Sigma_{2i} (\beta,t) - \Sigma_{2i} (\bar \beta,t)\right| = O_{\mathbb{P}}(b_n ).
 $$
It remains to uniformly bound $\Sigma_{ni}(\beta,t)$ and for this purpose  we use empirical process tools. Let us
  introduce some notation. Let $\mathcal{G}$ be a class of functions
  of the observations with envelope function $G$ and let
$$ J(\delta,\mathcal{G}, L^2 )=\sup_Q \int_0^\delta \sqrt{1+\ln N
  (\varepsilon \|G\|_{2},\mathcal{G}, L^2(Q) ) } d\varepsilon
,\qquad 0<\delta\leq 1,
$$ denote the uniform entropy integral, where the supremum is taken
  over all finitely discrete probability distributions $Q$ on the
  space of the observations, and $\| G \|_{2}$ denotes the norm of
  $G$ in $L^2(Q)$. Let $Z_1,\cdots,Z_n$ be a sample of independent
  observations and let
\begin{equation*}
\mathbb{G}_n g=\frac{1}{\sqrt{n}}\sum_{i=1}^n \gamma(Z_i) , \qquad \gamma \in\mathcal{G},
\end{equation*}
be the empirical process indexed by $\mathcal{G}$. If the covering
number $N (\varepsilon ,\mathcal{G}, L^2(Q) ) $ is of polynomial order
in $1/\varepsilon,$ there exists a constant $c>0$ such that
$J(\delta,\mathcal{G}, L^2 )\leq c \delta \sqrt{\ln(1/\delta)}$ for
$0<\delta<1/2.$ Now if $\mathbb{E}\gamma^2 < \delta^2 \mathbb{E}G^2$
for every $\gamma$ and some $0<\delta <1$, and
$\mathbb{E}G^{(4\upsilon-2)/(\upsilon-1)}<\infty$ for some
$\upsilon>1$, under mild additional measurability conditions that are satisfied in our context, Theorem
3.1 of \cite{Vaart2011} implies
\begin{equation}\label{vwww0}
 \sup_{\mathcal{G}}|\mathbb{G}_n \gamma| = J(\delta,\mathcal{G}, L^2
 )\left( 1 + \frac{ J(\delta^{1/\upsilon},\mathcal{G}, L^2 )}{\delta^2
   \sqrt{n} }
 \frac{\|G\|_{(4\upsilon-2)/(\upsilon-1)}^{2-1/\upsilon}}{\|G\|_{2}^{2-1/\upsilon}}
 \right)^{\upsilon/(2\upsilon-1)} \|G\|_2 O_{\mathbb{P}}(1),
\end{equation}
where $\|G\|_{2}^2 = \mathbb{E}G^2$ and the $ O_{\mathbb{P}}(1)$ term is
independent of $n.$ Note that the family $\mathcal{G}$ could change
with $n$, as soon as the envelope is the same for all $n$.  We apply
this result to the family of functions $\mathcal{G} = \{ \gamma (\cdot; \beta,w,t)- \gamma (\cdot; \bar\beta,w,t): t\in [0,1], \beta\in\mathcal{B},w \in\mathbb{R}\}$
where
$$
  \gamma (Y,X;\beta,w,t) = Y(t) L ((X^\prime \beta -w)g^{-1})
)
  $$
for a sequence $g$ that converges to
zero and the envelope $$G(Y,X)=\sup_{t\in[0,1]} |Y(t)| \sup_{w\in\mathbb{R}} L(w).$$ Its
entropy number is of polynomial order in $1/\varepsilon$,
  independently of $n$, as $L(\cdot)$ is of bounded variation and the families of indicator functions have polynomial complexity, see for
  instance \cite{Vaart1998}.  Now for any $\gamma \in \mathcal{G}$, $
  \mathbb{E} \gamma ^2 \leq C g \mathbb{E} G^2, $ for some
  constant $C$.  Let $\delta = g^{1/2},$ so that $ \mathbb{E} \gamma
  ^2 \leq C^\prime \delta^2 \mathbb{E} G^2, $ for some
  constant $C ^\prime$ and $\upsilon =3/2$, which corresponds to
  $\mathbb{E}G^{8}<\infty$ that is guaranteed by our assumptions.  Thus the
  bound in (\ref{vwww0})  yields
$$
 \sup_{\mathcal{G}}\left|\frac{1}{g \sqrt{n}} \; \mathbb{G}_n \gamma\right| =  \frac{ \ln^{1/2}(n)}{ \sqrt{n g}}
 \left[ 1 + n^{-1/2}g ^{-2/3}\ln^{1/2}(n) \right]^{3/4} O_{\mathbb{P}}(1) ,
$$
where the $ O_{\mathbb{P}}(1)$ term is independent of $n$.
Since $n g^{4/3}/\ln n  \rightarrow
\infty,$
$$
\max_{1\leq i\leq n}\sup_{t\in[0,1] } \sup_{\beta\in\mathcal{B}_n }|\Sigma_{1ni}(\beta,t) - \Sigma_{1ni}(\bar\beta,t) | = O_{\mathbb{P}}(n^{-1/2}g^{-1/2} \ln^{1/2}n).
$$
The second part of the statement is now obvious.
\end{proof}


\quad

\begin{lem}\label{deltas2}
Assume that the density $f_{\bar\beta}(\cdot)$  is Lipschitz. Then
$$
 \max_{1\leq i\leq n}\left| \frac {1}{n-1} \sum_{k\neq i} \frac{1}{g}L_{ik} (\bar \beta) - f_{\bar\beta } (X_i^\prime \bar \beta)\right| = O_{\mathbb{P}}(n^{-1/2}g^{-1/2} \ln^{1/2}n + g ).
$$
\end{lem}

\begin{proof}[Proof of Lemma \ref{deltas2}]
We can write
\begin{eqnarray*}
 \frac {1}{n-1} \sum_{k\neq i} \frac{1}{g}L_{ik} (\bar \beta)
- f_{\bar\beta } (X_i^\prime \bar \beta) &=& \frac {1}{n} \sum_{k=1}^n \left\{ g^{-1}L_{ik} (\bar \beta) - \mathbb{E}[g^{-1}L_{ik} (\bar \beta)\mid X_i] \right\}
\\
&&+\mathbb{E}[g^{-1}L_{ik} (\bar \beta)\mid X_i]- f_{\bar\beta } (X_i^\prime \bar \beta) + O (n^{-1}g^{-1}).
\end{eqnarray*}
By the empirical process arguments used in Lemma \ref{Deltas}, the sum on the right-hand side of the display is of rate
$O_{\mathbb{P}}(n^{-1/2}g^{-1/2} \ln^{1/2}n)$ uniformly with respect to $i.$ The Lipschitz property of $f_{\bar\beta } $ and the fact that $\int |vL(v)| dv<\infty$ guarantee that $$\max_{1\leq i\leq n}|\mathbb{E}[g^{-1}L_{ik} (\bar \beta)\mid X_i]- f_{\bar\beta } (X_i^\prime \bar \beta)| \leq C g$$ for some constant $C.$
\end{proof}

\quad

\begin{lem}\label{deltas3}
For any $t\in[0,1]$ let $Y_k(t),$ $1\leq k \leq n,$ be an independent sample from a random variable  $Y(t)$ defined like in the proof of Proposition \ref{as_equiv}.
Let $r(v;t,\bar\beta) = \mathbb{E}[Y(t)\mid X^\prime \bar{\beta} = v] ,$ $v\in\mathbb{R},$ and assume that  $r(\cdot;t,\bar\beta)$   is twice differentiable and the second derivative is bounded by a constant independent of $t$.
If $r^\prime (v;t,\bar\beta)$ is the first derivative of $r(\cdot;t,\bar\beta),$ then, for any $t\in[0,1],$
$$
\frac{1}{n-1}\sum\limits _{k\neq i} \{r( X_i^\prime \bar\beta;t,\bar\beta) - r( X_k^\prime \bar\beta ;t,\bar\beta) \}\frac{1}{g}L_{ik}(\bar{\beta}) =
r^\prime(X_i^\prime\bar\beta ;t,\bar\beta) g D_{1,ni} + g^2 D_{1,ni} (t),
$$
where $\max_{1\leq i\leq n}|D_{1,ni}| = n^{-1/2} g^{-1/2} \ln^{1/2} n$ and
$
 \max_{1\leq i\leq n}\sup_{t\in[0,1]} \left| D_{1,ni} (t) \right| = O_{\mathbb{P}}\left(1 \right).
$
\end{lem}

\begin{proof}[Proof of Lemma \ref{deltas3}]
By Taylor expansion
\begin{multline*}
\frac{1}{n-1}\sum\limits _{k\neq i} \{r( X_i^\prime \bar\beta;t,\bar\beta) - r( X_k^\prime \bar\beta ;t,\bar\beta) \}\frac{1}{g}L_{ik}(\bar{\beta})= r^\prime( X_i^\prime \bar\beta;t,\bar\beta)  \frac{1}{n}\sum\limits _{k=1}^n ( X_i - X_k)^\prime \bar\beta \frac{1}{g}L_{ik}(\bar{\beta})\\
+ \frac{1}{n}\sum\limits _{k=1}^n r^{\prime \prime} ( x_{ik}(t);t,\bar\beta) [( X_i - X_k)^\prime \bar\beta ]^2 \frac{1}{g}L_{ik}(\bar{\beta}),
\end{multline*}
where $r^{\prime \prime}$ stands for the second derivative with respect to $v$ and $x_{ik}(t)$ is a point between $X_i^\prime \bar\beta $ and  $X_k^\prime \bar\beta .$ Since $L(\cdot)$ is symmetric, by the empirical process arguments as in Lemma \ref{Deltas}
$$
\max_{1\leq i\leq n} \left| \frac{1}{n}\sum\limits _{k=1}^n \frac{( X_i - X_k)^\prime \bar\beta }{g} \frac{1}{g}L_{ik}(\bar{\beta}) \right| = O_{\mathbb{P}}(n^{-1/2} g^{-1/2} \ln^{1/2} n).
$$
The result follows taking absolute values in the last sum in the last display, using the boudedness of $r^{\prime \prime}$ and the fact that
$$
\max_{1\leq i\leq n}\left| \frac{1}{n}\sum\limits _{k=1}^n \frac{ [( X_i - X_k)^\prime \bar\beta ]^2}{g^2} \frac{1}{g}L_{ik}(\bar{\beta}) - f_{\bar\beta}(X_i ^\prime \bar\beta)\int_{\mathbb{R}} v^2 |L(v)|dv \right| = o_{\mathbb{P}}(1).
$$
\end{proof}

\quad


\begin{lem}\label{mon_tric}
Assume that $\mathbb{E}[\exp(a\|X\|)]<\infty$ for some $a>0.$ Moreover the kernels $K$ and $L$
are of bounded variation, differentiable except at most a  finite set of points, and $\int_{\mathbb{R}} |u K(u)|du <\infty.$ Let $\mathcal{B}_n$ be a subset in the parameter space such that the event defined in equation (\ref{eventE}) with $b_n\rightarrow 0$ and $b_n n^{1/2}/\ln n \rightarrow \infty$ has probability tending to 1.
Let
$$
K_{12}(\beta) = K((X_1 - X_2)^\prime \beta/h), \quad L_{12}(\beta) = L((X_1 - X_2)^\prime \beta/g)
$$
and $\phi (\beta) = \phi ((X_1 - X_2)^\prime \mathbf{A}(\beta)).$
 If the density $f_{\bar\beta}$ is Lipschitz with constant $C_{1,\bar\beta}$, then there exists a constant $C$ depending only on $K,$ $L,$ $\|f_{\bar\beta}\|_{\infty}$ and  $C_{1,\bar\beta}$ such that
\begin{equation}\label{eeqq1}
\mathbb{P}\left\{ \mathbb{E}\left[ \sup_{b\in\mathcal{B}_n} \left| K_{12}(\beta) \phi_{12}(\beta) - K_{12}(\bar\beta) \phi_{12}(\bar\beta) \right| \mid X_1 \right]  \leq C b_n h^{1/2}\right\}\rightarrow 1,
\end{equation}
\begin{equation}\label{eeqq2}
\mathbb{E} \left[ \sup_{b\in\mathcal{B}_n} \left| K_{12}(\beta) \phi_{12}(\beta) - K_{12}(\bar\beta) \phi_{12}(\bar\beta) \right| \right]  \leq C b_n h^{1/2},
\end{equation}
\begin{equation}\label{eeqq4}
\mathbb{P}\left\{ \mathbb{E} \left[ \sup_{b\in\mathcal{B}_n}  \left| L_{12}(\beta) - L_{12}(\bar\beta) \right|^2 \mid X_1\right]  \leq C  b_n g^{-1} \right\}\rightarrow 1
\end{equation}
\begin{equation}\label{eeqq5}
\mathbb{P}\left\{ \mathbb{E} \left[ \sup_{b\in\mathcal{B}_n}  \left| L_{13}(\beta) - L_{13}(\bar\beta) \right|^2 |K_{12}(\bar\beta)|^2 \mid X_2, X_3\right]  \leq C h b_n g^{-1} \right\}\rightarrow 1,
\end{equation}
and
\begin{equation}\label{eeqq3}
\mathbb{E} \left[ \sup_{b\in\mathcal{B}_n}  \left| L_{13}(\beta) - L_{13}(\bar\beta)  \right|^2 |K_{12}(\bar\beta)|^2 \phi_{12}^2(\bar\beta) \right]  \leq C h  b_n g^{-1},
\end{equation}
\;
\end{lem}

\medskip

In Lemma \ref{mon_tric}  we provide different bounds for $L(\cdot)$ and $K(\cdot)$ because the bandwidths $g$ and $h$  have to satisfy the condition $h/g^2\rightarrow 0.$ Hence we need less restrictive conditions on the range of $h$ if we want to allow for a larger domain for the pair $(g,h).$

\bigskip

\begin{proof}[Proof of Lemma \ref{mon_tric}]
Since the kernel $K$ is of bounded univariate kernels, let $K_1$ and $K_2$ non decreasing bounded functions such that $K = K_1 - K_2$ and denote $K_{1h}=K_1(\cdot/h)$. Clearly, it is sufficient to prove the result with $K_1$, similar arguments apply for $K_2$ and hence we get the results for $K$. For simpler writings we assume that $K$ is differentiable and let $K_1(x) = \int^x_{-\infty} [K^{\;\prime} (t) ]^+ dt$ and $K_2(x) = \int^x_{-\infty} [K^{\;\prime} (t) ]^- dt,$ $x\in\mathbb{R},$ Here $[K^{\;\prime}]^+$ (resp. $[K^{\;\prime}]^-$) denotes the positive (resp. negative) part of $K^{\;\prime}$. The general case where a finite set of nondifferentiability is allowed  can be handled with obvious modifications. Let $K_{1h} (t) = K_{1}(t/h) $ and recall that $Z_i(\beta) = X_i ^\prime \beta.$
Note that $|\exp(-t^2) - \exp(-s^2)|\leq \sqrt{2} |t-s|.$
For any $\beta\in\mathcal{B}_n$ and an elementary event in the set
$\mathcal{C}_n = \{ \max_{1\leq i\leq n} \|X_i\| \leq c\log n \}\subset \mathcal{E}_n$ for some large constant $c,$
\begin{multline*}
\left|K_{1h}\left(Z_{1}(\beta) -  Z_{2}(\beta)  \right) \phi_{12}(\beta) - K_{1h}\left( Z_{1} (\bar\beta)-  Z_{2} (\bar\beta) \right) \phi_{12}(\bar\beta) \right|\\
\leq  \sqrt{2} \;b_n K_{1h}( Z_{1}(\bar\beta) -  Z_{2} (\bar\beta)+ 2b_n ) \\
+
[  K_{1h}( Z_{1}(\bar\beta) -  Z_{2}(\bar\beta) + 2b_n ) - K_{1h}\left(Z_{1}(\bar\beta) - Z_{2} (\bar\beta) - 2b_n \right)] \phi_{12}(\bar\beta).
\end{multline*}
The upper bound on the left-hand side is uniform with respect to $\beta.$
By a suitable change of variable and since the density $f_{\overline{\beta}}$ is bounded, it is easy to check that
$$\mathbb{E}\left[ K_{1h}\left(Z_{1} (\bar\beta)- Z_2 (\bar\beta) +2b_n \right)\mid Z_1(\bar\beta) \right]$$ is bounded by a constant times $hb_n$.
Next, note that since $nh\rightarrow\infty,$ there exists a constant $C^\prime$ independent of $n$  such that on the set $\mathcal{C}_n$ we have $|Z_{1}(\bar\beta) - Z_2 (\bar\beta) \pm 2b_n|/h\leq C^\prime h^{-1/2}.$
Then, applying twice a  change of variables and using the Lipschitz property  of $f_{\bar\beta}$, on the set $\mathcal{C}_n,$
\begin{multline*}\label{int_df2}
\mathbb{E}\left[ \left|K_{1h}\left(Z_{1} (\bar\beta)- Z_2 (\bar\beta) +2b_n \right)
-  K_{1h}\left(Z_{1}(\bar\beta) - Z_2 (\bar\beta) - 2b_n\right)\right|  \mathbf{1}\{\mathcal{C}_n\} \mid Z_1(\bar\beta) \right]\\ \leq h \int_{[-C^{\;\prime} /h^{1/2} , \;C^{\;\prime} /h^{1/2}]}K_1 (u)\left|f_{\bar\beta} (2b_n + Z_1 (\bar\beta) -uh) -  f_{\bar\beta} (- 2b_n+ Z_1 (\bar\beta)-uh)\right| du\\
\leq h \times \sup_{t\in\mathbb{R}}\left|f_{\bar\beta} (2b_n + t) -  f_{\bar\beta} (- 2b_n+t)\right| \int_{[-C^{\;\prime} /h^{1/2} , \;C^{\;\prime} /h^{1/2}]}K_1 (u) du\\
\leq Ch^{1/2} b_n,
\end{multline*}
for some constant $C>0.$
Since by a suitable choice of $c$ the probability of $\mathbf{1}\{\mathcal{C}_n\}$ given $Z_1(\bar\beta)$ could be made smaller than any fixed negative power of $n,$ and the probability of the event $\{|Z_1(\bar \beta)|\leq c\log n \}$ could be also made very small,  the bound in the last display implies the statement (\ref{eeqq1}). For the statement (\ref{eeqq2}) it suffices to take expectation.

For the bound in equation (\ref{eeqq4}), recall that $L(t) = L(|t|)$ for any $t\in\mathbb{R}$ so that we can consider only nonnegative $t$. Moreover, without loss of generality we can consider $L$ nonnegative and decreasing on $[0,\infty),$ otherwise, since $L$ is of bounded variation, it could be written an the difference of two nonnegative decreasing functions on $[0,\infty).$
%
Moreover, let $Z_{13}(\beta)=|Z_{1}(\beta) - Z_{3}(\beta)|$  and $L_{g,13}(\beta) = L ( Z_{13}(\beta)/g). $ We split the problem in two cases: $Z_{13}(\beta)\leq Z_{13}(\bar\beta)$ and $Z_{13}(\beta) > Z_{13}(\bar\beta).$ Then, for $\beta\in\mathcal{B}_n$ and on the set $\mathcal{C}_n$  we have
\begin{multline*}
\left| L_{g,13}(\beta) - L_{g,13}(\bar\beta)  \right| \mathbf {1}\{Z_{13}(\beta)\leq Z_{13}(\bar\beta)\}\\
\leq [L(0) - L_{g,13}(\bar\beta)  ] \mathbf {1}\{Z_{13}(\beta)\leq Z_{13}(\bar\beta), Z_{13}(\bar\beta) \leq 2 b_n\}\\
+ [ L_{g,13}(\beta) - L_{g,13}(\bar\beta)  ] \mathbf {1}\{Z_{13}(\beta)\leq Z_{13}(\bar\beta), Z_{13}(\bar\beta) \geq  2 b_n\}\\
\leq C  b_n^2 g^{-2} \mathbf {1}\{ Z_{13}(\bar\beta) \leq 2 b_n\}
\\
+  \left[ L ( (Z_{13}(\bar\beta)-2b_n)/g) - L ( Z_{13}(\bar\beta)/g)\right]\mathbf {1}\{Z_{13}(\beta)\leq Z_{13}(\bar\beta),Z_{13}(\bar\beta) \geq 2 b_n\}\\
=C b_n^2 g^{-2} \mathbf {1}\{ Z_{13}(\bar\beta) \leq 2 b_n\} + A_n
\end{multline*}
and
\begin{multline*}
\left| L_{g,13}(\beta) - L_{g,13}(\bar\beta)  \right| \mathbf {1}\{Z_{13}(\beta) > Z_{13}(\bar\beta)\}\\
\leq  \left[ L ( Z_{13}(\bar\beta)/g) - L ( (Z_{13}(\bar\beta)+2b_n)/g)\right]\mathbf {1}\{Z_{13}(\beta) > Z_{13}(\bar\beta)\}=  B_n,
\end{multline*}
for some constant $C.$ Let us notice that
\begin{multline*}
A_n+B_n \leq  [L ( [Z_{13}(\bar\beta) - 2b_n]/g) -
L ( [Z_{13}(\bar\beta) + 2b_n]/g)] \mathbf {1}\{ Z_{13}(\bar\beta) \geq 2 b_n\}\\
+[L ( [Z_{13}(\bar\beta) ]/g) -
L ( [Z_{13}(\bar\beta) + 2b_n]/g)] \mathbf {1}\{ 0\leq Z_{13}(\bar\beta) \leq  2 b_n\}\\
\leq  [L ( [Z_{13}(\bar\beta) - 2b_n]/g) -
L ( [Z_{13}(\bar\beta) + 2b_n]/g)] \mathbf {1}\{ Z_{13}(\bar\beta) \geq 2 b_n\} \\+ C b_n^2 g^{-2} \mathbf {1}\{ Z_{13}(\bar\beta) \leq 2 b_n\}\\
\leq  [L ( [Z_{13}(\bar\beta) - 2b_n]/g) -
L ( [Z_{13}(\bar\beta) + 2b_n]/g)]  \\+ 2C b_n^2 g^{-2} \mathbf {1}\{ Z_{13}(\bar\beta) \leq 2 b_n\}\\
=D_n + 2C b_n^2 g^{-2} \mathbf {1}\{ Z_{13}(\bar\beta) \leq 2 b_n\}.
\end{multline*}
On the other hand,
$
0\leq D_n \leq 4b_n g ^{-1} | L^{\prime} (\widetilde Z) |
$
where $\widetilde Z$ is some value such that $|\widetilde Z - Z_{13}(\bar\beta)|\leq 2 b_ng^{-1}.$ Since, for some constant $c,$ $|L^\prime (v)|\leq c |v|$ in a neighborhood of the origin,
$$
D_n \leq 4b_n g ^{-1} | L^{\prime} (Z_{13}(\bar\beta) | + C^\prime b_n^2 g ^{-2},
$$
for some constant $C^\prime.$ Since $L^\prime$ is bounded, deduce that
$| L_{g,13}(\beta) - L_{g,13}(\bar\beta)  |^2$ is bounded by $C b_n^2 g ^{-1} | g^{-1} L^{\prime} (Z_{13}(\bar\beta) | + o(b_n^2 g ^{-1})$ for some constant $C$. Take conditional expectation given $X_1$, that is the same with the conditional expectation given $Z_1(\beta),$ and deduce the bound in equation (\ref{eeqq4}).

On the set of events $\mathcal{C}_n$,
$$
\sup_{\beta\in\mathcal{B}_n} \left| L_{13}(\beta) - L_{13}(\bar\beta)  \right| |K_{12}(\bar\beta)| \leq \{D_n + 3C b_n^2 g^{-2} \mathbf {1}\{ Z_{13}(\bar\beta) \leq 2 b_n\}\}|K_{12}(\bar\beta)|.
$$
Take conditional expectation and use  standard change of variables to derive the bound in equation (\ref{eeqq5}). Take expectation and remember that $\phi_{12}(\bar\beta)$ is bounded to derive the moment bound in equation (\ref{eeqq3}). \end{proof}

\quad

\begin{lem}\label{deltas4}
Under the conditions of Lemma \ref{Deltas}
 $$
\sup_{t\in [0,1]} \sup_{\beta\in\mathcal{B}_n } \max_{1\leq i\leq n}\left|\Sigma_{2i} (\beta,t) - \Sigma_{2i} (\bar \beta,t)\right| = O_{\mathbb{P}}(b_n ).
 $$
 \end{lem}

\begin{proof}[Proof of Lemma \ref{deltas4}]
We can write
\begin{multline*}
\left|\Sigma_{2i} (\beta,t) - \Sigma_{2i} (\bar \beta,t) \right| \leq \mathbb{E}\left[|Y(t)|
  \left|g^{-1} L\left( (X_i-X)^\prime \beta /g\right) - g^{-1} L\left( (X_i-X)^\prime \bar \beta /g\right)\right|\mid X_i\right]\\
  =  \mathbb{E}\left[ \mathbb{E}\left\{|Y(t)| X\right\}
  g^{-1} \left|L\left( (X_i-X)^\prime \beta /g\right) - L\left( (X_i-X)^\prime \bar \beta /g\right)\right|\mid X_i\right].
\end{multline*}
Now, we can apply the monotonicity argument we used in Lemma \ref{mon_tric} and deduce the bound. \end{proof}

\quad

\begin{lem}
\label{prop:I1} Under the conditions of Proposition \ref{sam_1mai}, $I_{1}(\beta_0)=o_{\mathbb{P}}\left(n^{-1}h^{-1/2}\right).$
\end{lem}

\begin{proof}[Proof of Lemma \ref{prop:I1}]
With the notation defined in equation (\ref{simpli2}) we have
\begin{align*}
I_{1}\left(\beta_{0}\right)=\dfrac{1}{n\left(n-1\right)^{3}g^{2}h}\sum_{i=1}^{n}\sum_{j\neq i}\sum_{k\neq i}\sum_{l\neq j}\left\langle \left(r_{i}-r_{k}\right)\left(\cdot;\,\beta_{0}\right),\,\left(r_{j}-r_{l}\right)\left(\cdot;\,\beta_{0}\right)\right\rangle _{L^{2}} L_{ik}L_{jl}K_{ij}\phi_{ij}
\end{align*}
 and if we denote by $I_{1,1}\left(\beta_{0}\right)$ the term where
$i$, $j$, $k$ and $l$ are all different, then
\begin{align*}
\mathbb{E}\left[I_{1,1}\left(\beta_{0}\right)\right]=\dfrac{\left(n-2\right)\left(n-3\right)}{\left(n-1\right)^{2}g^{2}h}\mathbb{E}\left[\left\langle \mathbb{E}\left[\left(r_{i}-r_{k}\right)\left(\cdot;\,\beta_{0}\right)L_{ik}\mid Z_{i}\left(\beta_{0}\right)\right],\right.\right.\qquad\qquad\qquad\\
\left.\left.\mathbb{E}\left[\left(r_{j}-r_{l}\right)\left(\cdot;\,\beta_{0}\right)L_{jl}\mid Z_{j}\left(\beta_{0}\right)\right]\right\rangle _{L^{2}}K_{ij}\phi_{ij}\right]=O\left(g^{4}\right)
\end{align*}
as soon as $g^{-1}\mathbb{E}\left[\left(r_{i}-r_{k}\right)\left(t;\,\beta_{0}\right)L_{ik}\left(\beta_{0}\right)\mid Z_{i}\left(\beta_{0}\right)\right]=O\left(g^{2}\right)D\left(t;Z_{i}\left(\beta_{0}\right)\right)$
with $D\left(\cdot\right)$ bounded, which is guaranteed by Assumption \ref{ass_app}-(c). When $i$, $j$, $k$ and
$l$ take no more than $3$ different values, the number of terms
is reduced by a factor $n$, and thus we have that $\mathbb{E}\left[I_{1,2}\left(\beta_{0}\right)\right]=O\left(n^{-1}g^{-1}\right)=
o\left(n^{-1}h^{-1/2}\right)$.
Similar reasoning can be applied to prove that $\mathbb{E}\left[I_{1}^{2}\left(\beta_{0}\right)\right]=o\left(n^{-2}h^{-1}\right)$.
See also Proposition A.1. in \cite{FanLi1996}.
\end{proof}

\quad

\begin{lem}
Under the conditions of Proposition \ref{sam_1mai},
\label{prop:I3}$I_{3} (\beta_0) =o_{\mathbb{P}}\left(n^{-1}h^{-1/2}\right)$ .\end{lem}

\begin{proof}[Proof of Lemma \ref{prop:I3}]
Write
\begin{eqnarray*}
I_{3}\left(\beta_{0}\right) & = & \dfrac{1}{n\left(n-1\right)^{3}g^{2}h}\sum_{i=1}^{n}\sum_{j\neq i}\sum_{k\neq i}\sum_{l\neq j}\left\langle \epsilon_{k}\left(\cdot\right),\,\epsilon_{l}\left(\cdot\right)\right\rangle _{L^{2}}L_{ik} L_{jl} K_{ij} \phi_{ij} \\
 & = & \dfrac{1}{n\left(n-1\right)^{3}g^{2}h}\sum_{i=1}^{n}\sum_{j\neq i}\sum_{k\neq i}\sum_{l\neq j,k}\left\langle \epsilon_{k}\left(\cdot\right),\,\epsilon_{l}\left(\cdot\right)\right\rangle _{L^{2}}L_{ik} L_{jl} K_{ij} \phi_{ij} \\
 &  & +\dfrac{1}{n\left(n-1\right)^{3}g^{2}h}\sum_{i=1}^{n}\sum_{j\neq i}\sum_{k\neq i,j}\left\Vert \epsilon_{k}\left(\cdot\right)\right\Vert _{L^{2}}^{2}L_{ik} L_{ji} K_{ij} \phi_{ij} \\
 &  & +\dfrac{1}{n\left(n-1\right)^{3}g^{2}h}\sum_{i=1}^{n}\sum_{j\neq i}\left\Vert \epsilon_{j}\left(\cdot\right)\right\Vert _{L^{2}}^{2}L_{ij} L_{ji} K_{ij} \phi_{ij} \\
 & = & I_{3,1}\left(\beta_{0}\right) +I_{3,2}\left(\beta_{0}\right)+I_{3,3}\left(\beta_{0}\right).
\end{eqnarray*}
Then
\begin{eqnarray*}
\mathbb{E}\left[I_{3,1}\left(\beta_{0}\right)\right] & = & \dfrac{1}{\left(n-1\right)^{2}g^{2}h}\mathbb{E}\left[\left\langle \epsilon_{1}\left(\cdot\right),\,\epsilon_{2}\left(\cdot\right)\right\rangle _{L^{2}}L_{12}^{2} K_{12} \phi_{12} \right]\\
 & = & O\left(n^{-2}g^{-2}\right)\mathbb{E}\left[\left|\left\langle \epsilon_{1}\left(\cdot\right),\,\epsilon_{2}\left(\cdot\right)\right\rangle _{L^{2}}h^{-1}K_{12} \right|\right]\\
 & = & O\left(n^{-2}g^{-2}\right),
\end{eqnarray*}
 $\mathbb{E}\left[I_{3,2}\left(\beta_{0}\right)\right]=O\left(n^{-1}g^{-1}\right)$
and $\mathbb{E}\left[I_{3,3}\left(\beta_{0}\right)\right]=O\left(n^{-2}g^{-2}\right)$, thus $\mathbb{E}\left[I_{3}\left(\beta_{0}\right)\right]=o\left(n^{-1}h^{-1/2}\right).$
By quite straightforward but tedious calculations, it can be  proved that $\mathbb{E}\left[I_{3}^{2}\left(\beta_{0}\right)\right]=o\left(n^{-2}h^{-1}\right)$ and the rate of $I_{3} (\beta_0)$ follows.
\end{proof}

\quad

\begin{lem}
\label{lem:CondVar} Let $A_{n}\left(\beta_{0}\right)$ and $B_{n}\left(\beta_{0}\right)$ be defined as in equation (\ref{an_bn}). Under the conditions of Proposition \ref{sam_1mai}, $A_{n}\left(\beta_{0}\right)\rightarrow 1$ and
$B_{n}\left(\beta_{0}\right)\rightarrow 0$ in probability, and
\[
\forall\varepsilon>0,\quad\sum_{i=2}^{n}\mathbb{E}\left[G_{n,i}^{2}I\left(\left|G_{n,i}\right|>\varepsilon\right)\mid\mathcal{F}_{n,i-1}\right]\rightarrow0 ,\quad\text{in probability}.
\]
\end{lem}
\begin{proof}[Proof of Lemma  \ref{lem:CondVar}]
First, we have
\begin{multline*}
\mathbb{E}\left[A_{n}\left(\beta_{0}\right)\right] =  \mathbb{E}\left[\mathbb{E}\left[A_{n}\left(\beta_{0}\right)\mid X_{1},\dots,X_{n}\right]\right]\\
  =  \mathbb{E}\left[\dfrac{2n}{\omega_{n}^{2}\left(\beta_{0}\right)
  \left(n-1\right)h}\int\int\Gamma\left(s,t\right)\hat{f}_{\beta_{0},i}^{2}
  \mathbb{E}\left[\epsilon_{j}\left(s\right)\epsilon_{j}\left(t\right)\right]
  \hat{f}_{\beta_{0},j}^{2}K_{ij}^{2}\phi_{ij}^{2}ds\, dt\right]\\
 =  \dfrac{n}{n-1}\xrightarrow{n\to\infty}1.
\end{multline*}
Moreover,
\begin{eqnarray*}
\mbox{Var}\left(A_{n}\left(\beta_{0}\right)\right) & \leq & \dfrac{64\left\Vert \phi\right\Vert _{\infty}^{4}}{\left(n-1\right)^{4}h^{2}}\sum_{i=3}^{n}\sum_{j=2}^{i-1}\sum_{j^{\prime}=1}^{j-1}\mathbb{E}\left[\omega_{n}^{-2}\left(\beta_{0}\right)\hat{f}_{\beta_{0},i}^{4}\hat{f}_{\beta_{0},j}^{2}\hat{f}_{\beta_{0},j^{\prime}}^{2}K_{ij}^{2} K_{ij^{\prime}}^{2} \right]\\
 &  & \hphantom{\dfrac{64\left\Vert \phi\right\Vert _{\infty}^{4}}{\left(n-1\right)^{4}h^{2}}\sum_{i=3}^{n}}\times\int\int\int\int\Gamma^2\left(s,t\right)\Gamma^2\left(u,v\right)dsdtdudv\\
 &  & +\dfrac{32\left\Vert \phi\right\Vert _{\infty}^{4}}{\left(n-1\right)^{4}h^{2}}\sum_{i=3}^{n}\sum_{i^{\prime}=2}^{i-1}\sum_{j=1}^{i^{\prime}-1}\mathbb{E}\left[\omega_{n}^{-2}\left(\beta_{0}\right)\hat{f}_{\beta_{0},i}^{2}\hat{f}_{\beta_{0},i^{\prime}}^{2}\hat{f}_{\beta_{0},j}^{4}K_{ij}^{2} K_{i^{\prime}j}^{2} \right]\\
 &  & \hphantom{+\dfrac{32\left\Vert \phi\right\Vert _{\infty}^{4}}{\left(n-1\right)^{4}}\sum_{i=3}^{n}}\times\int\int\int\int\Gamma\left(s,t\right)\Gamma\left(u,v\right)\mathfrak{G}\left(s,t,u,v\right)dsdtdudv\\
 &  & +\dfrac{16\left\Vert \phi\right\Vert _{\infty}^{4}}{\left(n-1\right)^{4}h^{2}}\sum_{i=3}^{n}\sum_{i^{\prime}=2}^{i-1}\sum_{j=1}^{i^{\prime}-1}\mathbb{E}\left[\omega_{n}^{-2}\left(\beta_{0}\right)\hat{f}_{\beta_{0},i}^{4}\hat{f}_{\beta_{0},j}^{4}K_{ij}^{4} \right]\\
 &  & \hphantom{+\dfrac{16\left\Vert \phi\right\Vert _{\infty}^{4}}{\left(n-1\right)^{4}}\sum_{i=3}^{n}}\times\int\int\int\int\Gamma\left(s,t\right)\Gamma\left(u,v\right)\mathfrak{G}\left(s,t,u,v\right)dsdtdudv\\
 & = & o\left(n^{-1}h^{-1/2}\right),
\end{eqnarray*}
where $\mathfrak{G}\left(s,t,u,v\right)=\mathbb{E}\left[\epsilon\left(s\right)\epsilon\left(t\right)\epsilon\left(u\right)\epsilon\left(v\right)\right]$.
The decomposition of  $\mathbb{E}\left[B_{n}^{2}\right]$ involves the same
type of terms and is therefore also of rate $o\left(n^{-1}h^{-1/2}\right)$.
For the Lindeberg condition, we have $\forall\varepsilon>0$, $\forall n\geq1$
and $1<i\leq n$
\[
\mathbb{E}\left[G_{n,i}^{2}I\left(\left|G_{n,i}\right|>\varepsilon\right)\mid\mathcal{F}_{n,i-1}\right]\leq\dfrac{\mathbb{E}\left[G_{n,i}^{4}\mid\mathcal{F}_{n,i-1}\right]}{\varepsilon^{2}}.
\]
Then
\begin{align*}
 &\hspace{-2cm} \sum_{i=2}^{n}\mathbb{E}\left[G_{n,i}^{2}I\left(\left|G_{n,i}\right|>\varepsilon\right)\mid\mathcal{F}_{n,i-1}\right]\\
\leq\; & \dfrac{1}{\varepsilon^{2}}\sum_{i=2}^{n}\mathbb{E}\left[G_{n,i}^{4}\mid\mathcal{F}_{n,i-1}\right]\\
\leq\; & \dfrac{1}{\varepsilon^{2}}\dfrac{16}{\left(n-1\right)^{4}h^{2}}\sum_{i=2}^{n}\int\int\int\int\mathfrak{G}\left(s_{1},s_{2},s_{3},s_{4}\right)\hat{f}_{\beta_{0},i}^{4}\\
 & \hphantom{\dfrac{1}{\varepsilon^{2}}\dfrac{16}{\left(n-1\right)^{4}h^{2}}\sum_{i=2}^{n}\int\int\int\int}\times\prod_{k=1}^{4}\sum_{j_{k}=1}^{i-1}\epsilon_{j_{k}}\left(s_{k}\right)\hat{f}_{\beta_{0},j_{k}}K_{ij_{k}}\phi_{ij_{k}}ds_{k}.
\end{align*}
The expectation of the last majorant is of rate
\begin{align*}
 & \hspace{-2cm} O\left(n^{-1}\right)\int\int\int\int\mathfrak{G}\left(s_{1},s_{2},s_{3},s_{4}\right)\Gamma\left(s_{1},s_{2}\right)\Gamma\left(s_{3},s_{4}\right)ds_{1}ds_{2}ds_{3}ds_{4}\qquad\\
 & \times\mathbb{E}\left[\hat{f}_{\beta_{0},i}^{4}\hat{f}_{\beta_{0},j}^{2}\hat{f}_{\beta_{0},j^{\prime}}^{2}h^{-1}K_{ij}^{2}h^{-1}K_{ij^{\prime}}^{2}\phi_{ij}^{2}\phi_{ij^{\prime}}^{2}\right]\\
 & +O\left(n^{-2}h^{-1}\right)\sum_{i=2}^{n}\int\int\int\int\mathfrak{G}^{2}\left(s_{1},s_{2},s_{3},s_{4}\right)ds_{1}ds_{2}ds_{3}ds_{4}\qquad\\
 & \times\mathbb{E}\left[\hat{f}_{\beta_{0},i}^{4}\hat{f}_{\beta_{0},j}^{4}h^{-1}K_{ij}^{4}\phi_{ij}^{4}\right]\\
= & o\left(n^{-1}h^{-1/2}\right).
\end{align*}
\end{proof}

\quad

\begin{lem}
\label{lem:VarLimit}
Under the conditions of Proposition \ref{sam_1mai}, $\omega_{n}^{2}\left(\beta_{0}\right)\rightarrow\omega^{2}\left(\beta_{0}\right)>0,$ in probability.\end{lem}

\begin{proof}[Proof of Lemma \ref{lem:VarLimit}] We have
\[
\mathbb{E}\left[\omega_{n}^{2}\left(\beta_{0}\right)\right]=2\mathbb{E}\left[\hat{f}_{\beta,i}\hat{f}_{\beta,j}h^{-1}K_{ij}^{2}\left(\beta\right)\phi_{ij}^{2}\left(\beta\right)\right]\times\intop\intop\Gamma^{2}\left(s,t\right)ds\, dt.
\]
On the other hand,
\begin{align*}
 & \mathbb{E}\left[\hat{f}_{\beta,i}\hat{f}_{\beta,j}h^{-1}K_{ij}^{2} \phi_{ij}^{2} \right]\\
= & \dfrac{1}{g^{2}h}\mathbb{E}\left[\sum_{k\neq i}\sum_{l\neq j}\sum_{k^{\prime}\neq i}\sum_{l^{\prime}\neq j}L_{ik}L_{jl}L_{ik^{\prime}}L_{jl^{\prime}}h^{-1}K_{ij}^{2} \phi_{ij}^{2} \right]\\
= & \dfrac{1}{g^{2}h\left(n-1\right)^{2}}\mathbb{E}\left[\sum_{k\neq i}\sum_{l\neq j}\sum_{k^{\prime}\neq i}\sum_{l^{\prime}\neq j}L_{ik}L_{jl}L_{ik^{\prime}}L_{jl^{\prime}}h^{-1}K_{ij}^{2} \phi_{ij}^{2} \right]\\
= & \dfrac{1}{g^{2}h\left(n-1\right)^{4}}\mathbb{E}\left[\sum_{k\neq i}\sum_{l\neq j}\sum_{k^{\prime}\neq i}\sum_{l^{\prime}\neq j}L_{ik}L_{jl}L_{ik^{\prime}}L_{jl^{\prime}}h^{-1}K_{ij}^{2} \phi_{ij}^{2} \right]\\
 & +o\left(n^{-1}h^{-1/2}\right)\\
= & \dfrac{\left(n-1\right)^{3}}{\left(n-2\right)\left(n-3\right)\left(n-4\right)}\tilde{\omega}_{n}^{2}\left(\beta_{0}\right)+o\left(n^{-1}h^{-1/2}\right)
\end{align*}
where
\begin{align*}
\tilde{\omega}_{n}^{2} \left(\beta_{0}\right) =\; & \mathbb{E}\left[\int\int\int\int\dfrac{1}{g}L\left(\dfrac{z_{i}-z_{k}}{g}\right)\dfrac{1}{g}L\left(\dfrac{z_{j}-z_{l}}{g}\right)\dfrac{1}{g}L\left(\dfrac{z_{i}-z_{k^{\prime}}}{g}\right)\dfrac{1}{g}L\left(\dfrac{z_{j}-z_{l^{\prime}}}{g}\right)\right.\\
 & \qquad\qquad\times\dfrac{1}{h}K^{2}\left(\dfrac{z_{i}-z_{j}}{h}\right)\phi_{ij} \\
 & \qquad\qquad\times f_{\beta_0}\left(z_{k}\right)f_{\beta_0}\left(z_{l}\right)f_{\beta_0}\left(z_{k^{\prime}}\right)f_{\beta_0}\left(z_{l^{\prime}}\right)\\
 & \qquad\qquad\left.\times\pi_{\beta_0}\left(z_{i}\mid W_{i}\left(\beta_0\right)\right)\pi_{\beta_0}\left(z_{j}\mid W_{j}\left(\beta_0\right)\right)dz_{i}dz_{j}dz_{k}dz_{l}dz_{k^{\prime}}dz_{l^{\prime}}\vphantom{\dfrac{1}{g}}\right]\\
=\; & \mathbb{E}\left[\int\int\int\int f_{\beta_0}\left(z_{i}-gs_{1}\right)f_{\beta_0}\left(z_{i}-gs_{2}\right)f_{\beta_0}\left(z_{j}-gt_{1}\right)f_{\beta_0}\left(z_{j}-gt_{2}\right)\right.\\
 & \qquad\qquad\times\pi_{\beta_0}\left(z_{i}\mid W_{i}\left(\beta_0\right)\right)\pi_{\beta_0}\left(z_{j}\mid W_{j}\left(\beta_0\right)\right)\phi_{ij}\\
 & \qquad\qquad\left.\times L\left(s_{1}\right)L\left(t_{1}\right)L\left(s_{2}\right)L\left(t_{2}\right)\dfrac{1}{h}K^{2}\left(\dfrac{z_{i}-z_{j}}{h}\right)dz_{i}dz_{j}ds_{1}dt_{1}ds_{2}dt_{2}\vphantom{\dfrac{1}{g}}\right]\\
=\; & \mathbb{E}\left[\int\int\int\int f_{\beta_0}\left(z_{i}-gs_{1}\right)f_{\beta_0}\left(z_{i}-gs_{2}\right)f_{\beta_0}\left(z_{i}-gu-gt_{1}\right)f_{\beta_0}\left(z_{i}-gu-gt_{2}\right)\right.\\
 & \qquad\qquad\times\pi_{\beta_0}\left(z_{i}\mid W_{i}\left(\beta_0\right)\right)\pi_{\beta_0}\left(z_{i}-gu\mid W_{j}\left(\beta_0\right)\right)\phi_{ij}\\
 & \qquad\qquad\left.\times L\left(s_{1}\right)L\left(t_{1}\right)L\left(s_{2}\right)L\left(t_{2}\right)K^{2}\left(u\right)dz_{i}ds_{1}dt_{1}ds_{2}dt_{2}du\vphantom{\dfrac{1}{g}}\right]\\
\to\; & \mathbb{E}\left[\int f_{\beta_0}^{4}\left(z\right)\pi_{\beta_0}\left(z\mid W_{i}\left(\beta_0\right)\right)\pi_{\beta_0}\left(z\mid W_{j}\left(\beta_0\right)\right)\phi_{ij}dz\right]\times\int K^{2}\left(u\right)du
\end{align*}
where the limit is obtained by standard arguments, using uniform continuity
of $f_{\beta_0}\left(\cdot\right)$ and $\pi_{\beta_0}\left(\cdot\mid w\right)$.\end{proof}
%
%

\end{document}